\documentclass[twoside,11pt]{article}
\usepackage{isipta}

\usepackage[british]{babel}

\usepackage{hyperref,url}

\usepackage{mathptmx}
\usepackage{amsmath}
\usepackage{courier}
\usepackage{amssymb}

\usepackage{enumerate}
\usepackage{enumitem,multicol}
\usepackage{tikz}
\usepackage{nicefrac}
\usepackage{bm}
\usepackage{algorithm}
\usepackage{algorithmicx}
\usepackage{algpseudocode}

\usepackage{graphicx}

\DeclareMathAlphabet\mathbfcal{OMS}{cmsy}{b}{n}

\algrenewcommand\algorithmicrequire{\textbf{Input:}}
\algrenewcommand\algorithmicensure{\textbf{Output:}}

\newcommand{\nats}{\mathbb{N}}
\newcommand{\natswith}{\nats_{0}}
\newcommand{\reals}{\mathbb{R}}

\newcommand{\realspos}{\reals_{>0}}
\newcommand{\realsnonneg}{\reals_{\geq 0}}

\newcommand{\states}{\mathcal{X}}
\newcommand{\observs}{\mathcal{Y}}

\newcommand{\lexp}{\underline{\mathbb{E}}_{\rateset,\mathcal{M}}}
\newcommand{\uexp}{\overline{\mathbb{E}}_{\rateset,\mathcal{M}}}

\newcommand{\gambles}{\mathcal{L}}
\newcommand{\gamblesX}{\gambles(\states)} 

\newcommand{\ind}[1]{\mathbb{I}_{#1}}

\newcommand{\rateset}{\mathcal{Q}}
\newcommand{\lrate}{\underline{Q}}

\newcommand{\asa}{\Leftrightarrow}

\newcommand{\norm}[1]{\left\lVert #1 \right\rVert}
\newcommand{\abs}[1]{\left\vert #1 \right\vert}

\newcommand{\coloneqq}{:\!=}

\def\presuper#1#2%
  {\mathop{}%
   \mathopen{\vphantom{#2}}^{#1}%
   \kern-\scriptspace%
   #2}

\makeatletter
\newcommand{\customlabel}[2]{%
   \protected@write \@auxout {}{\string \newlabel {#1}{{#2}{\thepage}{#2}{#1}{}} }%
   \hypertarget{#1}{\emph{#2}\!}
}
\makeatother

\usepackage{color,soul,booktabs}
\setulcolor{blue}
\newcommand{\BibTeX}{\textsc{B\kern-0.1emi\kern-0.017emb}\kern-0.15em\TeX}

\ShortHeadings{Efficient Computation of Updated Lower Expectations for ICTHMC's}{Krak et al.}

\begin{document}

\title{Efficient Computation of Updated Lower Expectations for\\ Imprecise Continuous-Time Hidden Markov Chains}
\author{\name Thomas Krak \email t.e.krak@uu.nl\\
\addr $^\dagger$Department of Information and Computing Sciences\\
Utrecht University (The Netherlands)
\AND
\name Jasper De Bock \email jasper.debock@ugent.be\\
\addr Department of Electronics and Information Systems, imec, IDLab\\
Ghent University (Belgium)
\AND
\name Arno Siebes$^\dagger$ \email a.p.j.m.siebes@uu.nl
}
\maketitle
\vspace{-7pt}

\begin{abstract}We consider the problem of performing inference with \emph{imprecise continuous-time hidden Markov chains}, that is, \emph{imprecise continuous-time Markov chains} that are augmented with random \emph{output} variables whose distribution depends on the hidden state of the chain. The prefix `imprecise' refers to the fact that we do not consider a classical continuous-time Markov chain, but replace it with a robust extension that allows us to represent various types of model uncertainty, using the theory of \emph{imprecise probabilities}. The inference problem amounts to computing lower expectations of functions on the state-space of the chain, given observations of the output variables.
We develop and investigate this problem with very few assumptions on the output variables; in particular, they can be chosen to be either discrete or continuous random variables. Our main result is a polynomial runtime algorithm to compute the lower expectation of functions on the state-space at any given time-point, given a collection of observations of the output variables.\vspace{-3pt}
\end{abstract}

\section{Introduction}\label{sec:introduction}%\vspace{-5pt}

A continuous-time Markov chain (CTMC) is a stochastic model that describes the evolution of a dynamical system under uncertainty. Specifically, it provides a probabilistic description of how such a system might move through a finite state-space, as time elapses in a continuous fashion. There are various ways in which this model class can be extended.

One such extension are continuous-time \emph{hidden} Markov chains (CTHMC's)~\citep{wei2002continuous}. Such a CTHMC is a stochastic model that contains a continuous-time Markov chain as a latent variable---that is, the actual realised behaviour of the system cannot be directly observed. This model furthermore incorporates random \emph{output} variables, which depend probabilistically on the current state of the system, and it is rather realisations of these variables that one observes. Through this stochastic dependency between the output variables and the states in which the system might be, one can perform inferences about quantities of interest that depend on these states---even though they have not been, or cannot be, observed directly.

Another extension of CTMC's, arising from the theory of \emph{imprecise probabilities}~\citep{Walley:1991vk}, are \emph{imprecise continuous-time Markov chains} (ICTMC's)~\citep{Skulj:2015cq, krak2016ictmc}. This extension can be used to robustify against uncertain numerical parameter assessments, as well as the simplifying assumptions of time-homogeneity and that the model should satisfy the Markov property. 
Simply put, an ICTMC is a \emph{set} of continuous-time stochastic processes, some of which are ``traditional'' time-homogeneous CTMC's. However, this set also contains more complicated processes, which are non-homogeneous and do not satisfy the Markov property.

In this current work, we combine these two extensions by considering \emph{imprecise continuous-time hidden Markov chains}---a stochastic model analogous to a CTHMC, but where the latent CTMC is replaced by an ICTMC. We will focus in particular on practical aspects of the corresponding inference problem. That is, we provide results on how to efficiently compute lower expectations of functions on the state-space, given observed realisations of the output variables. 

%Throughout, all results are stated without proof. We have made available an extended version of this work~\citep{krak2017icthmc}, which includes an appendix containing the proofs of all our results.
The proofs of all results are gathered in an appendix, where they are largely ordered by their chronological appearance in the main text.

\subsection{Related Work}\label{sec:related}

As should be clear from the description of CTHMC's in Section~\ref{sec:introduction}, this model class extends the well-known (discrete-time) \emph{hidden Markov models} (HMM's) to a continuous-time setting. In the same sense, the present subject of ICTHMC's can be seen to extend previous work on \emph{imprecise hidden Markov models} (iHMM's)~\citep{deCooman:2010gd} to a continuous-time setting. Hence, the model  under consideration should hopefully be intuitively clear to readers familiar with (i)HMM's. 

The main novelty of this present work is therefore not the (somewhat obvious) extension of iHMM's to ICTHMC's, but rather the application of recent results on ICTMC's~\citep{krak2016ictmc} to derive an efficient solution to the continuous-time analogue of inference in iHMM's. The algorithm that we present is largely based on combining these results with the ideas behind the MePiCTIr algorithm~\citep{deCooman:2010gd} for inference in credal trees under epistemic irrelevance.

A second novelty of the present paper is that, contrary to most of the work in the literature on iHMM's, we allow the output variables of the ICTHMC to be either discrete or continuous. This allows the model to be applied to a much broader range of problems. At the same time, it turns out that this does not negatively influence the efficiency of the inference algorithm.

\section{Preliminaries}\label{sec:prelim}

We denote the reals as $\reals$, the non-negative reals as $\realsnonneg$, and the positive reals as $\realspos$. The natural numbers are denoted by $\nats$, and we define $\natswith\coloneqq\nats\cup\{0\}$.

Since we are working in a continuous-time setting, a \emph{time-point} is an element of $\realsnonneg$, and these are typically denoted by $t$ or $s$. We also make extensive use of non-empty, finite sequences of time points $u\subset\realsnonneg$. These are taken to be ordered, so that they may be written $u=t_0,\ldots,t_n$, for some $n\in\natswith$, and such that then $t_i<t_j$ for all $i,j\in\{0,\ldots,n\}$ for which $i< j$. Such sequences are usually denoted by $u$ or $v$, and we let $\mathcal{U}$ be the entire set of them.

Throughout, we consider some fixed, finite state space $\states$. A generic element of $\states$ will be denoted by $x$. When considering the state-space at a specific time $t$, we write $\states_t\coloneqq\states$, and $x_t$ denotes a generic state-assignment at this time. When considering multiple time-points $u$ simultaneously, we define the joint state-space as $\states_u\coloneqq\prod_{t_i\in u}\states_{t_i}$, of which $x_u=(x_{t_0},\ldots,x_{t_n})$ is a generic element.

For any $u\in\mathcal{U}$, we let $\gambles(\states_u)$ be the set of all real-valued functions on $\states_u$.

\subsection{Imprecise Continuous-Time Markov Chains}\label{subsec:ictmc}

We here briefly recall the most important properties of imprecise continuous-time Markov chains (ICTMC's), following the definitions and results of~\citet{krak2016ictmc}. For reasons of brevity, we provide these definitions in a largely intuitive, non-rigorous manner, and refer the interested reader to this earlier work for an in-depth treatise on the subject.

An ICTMC will be defined below as a specific set of \emph{continuous-time stochastic processes}. Simply put, a continuous-time stochastic process is a joint probability distribution over random variables $X_t$, for each time $t\in\realsnonneg$, where each random variable $X_t$ takes values in $\states$.

It will be convenient to have a way to numerically parameterise such a stochastic process $P$. For this, we require two different kinds of parameters. First, we need the specification of the initial distribution $P(X_0)$ over the state at time zero; this simply requires the specification of some probability mass function on $\states_0$. Second, we need to parameterise the dynamic behaviour of the model.

In order to describe this dynamic behaviour, we require the concept of a \emph{rate matrix}. Such a rate matrix $Q$ is a real-valued $\lvert\states\rvert\times\lvert\states\rvert$ matrix, whose off-diagonal elements are non-negative, and whose every row sums to zero---thus, the diagonal elements are non-positive. Such a rate matrix may be interpreted as describing the ``rate of change'' of the conditional probability $P(X_s\,\vert\,X_t,X_u=x_u)$, when $s$ is close to $t$. In this conditional probability, it is assumed that $u<t$, whence the state assignment $x_u$ is called the \emph{history}. For small enough $\Delta\in\realspos$, we may now write that
\begin{equation*}
P(X_{t+\Delta}\,\vert\,X_t,X_u=x_u) \approx \bigl[I + \Delta Q_{t,x_u}\bigr](X_t, X_{t+\Delta})\,,
\end{equation*}
for some rate matrix $Q_{t,x_u}$, where $I$ denotes the $\lvert\states\rvert\times\lvert\states\rvert$ identity matrix, and where the quantity $[I + \Delta Q_{t,x_u}](X_t,X_{t+\Delta})$ denotes the element at the $X_t$-row and $X_{t+\Delta}$-column of the matrix $I + \Delta Q_{t,x_u}$. Note that in general, this rate matrix $Q_{t,x_u}$ may depend on the specific time $t$ and history $x_u$ at which this relationship is stated. 

If these rate matrices only depend on the time $t$ and not on the history $x_u$, i.e. if $Q_{t,x_u}=Q_t$ for all $t$ and all $x_u$, then it can be shown that $P$ satisfies the \emph{Markov property}: $P(X_s\,\vert\,X_t,X_u)=P(X_s\,\vert\,X_t)$. In this case, $P$ is called a \emph{continuous-time Markov chain}.

Using this method of parameterisation, an \emph{imprecise continuous-time Markov chain} (ICTMC) is similarly parameterised using a \emph{set} of rate matrices $\rateset$, and a \emph{set} of initial distributions $\mathcal{M}$. The corresponding ICTMC, denoted by $\mathbb{P}_{\rateset,\mathcal{M}}$, is the set of all continuous-time stochastic processes whose dynamics can be described using the elements of $\rateset$, and whose initial distributions are consistent with $\mathcal{M}$. That is, $\mathbb{P}_{\rateset,\mathcal{M}}$ is the set of stochastic processes $P$ for which $P(X_0)\in\mathcal{M}$ and for which $Q_{t,x_u}\in\rateset$ for every time $t$ and history $x_u$.

The \emph{lower expectation} with respect to this set $\mathbb{P}_{\rateset,\mathcal{M}}$ is then defined as
\begin{equation*}
\underline{\mathbb{E}}_{\rateset,\mathcal{M}}[\cdot\,\vert\,\cdot] \coloneqq \inf\left\{ \mathbb{E}_P[\cdot\,\vert\,\cdot]\,:\, P\in\mathbb{P}_{\rateset,\mathcal{M}} \right\}\,,
\end{equation*}
where $\mathbb{E}_P[\cdot\,\vert\,\cdot]$ denotes the expectation with respect to the (precise) stochastic process $P$. The \emph{upper expectation} $\uexp$ is defined similarly, and is derived through the well-known conjugacy property $\uexp[\cdot\,\vert\,\cdot] = -\lexp[-\cdot\,\vert\,\cdot]$. Note that it suffices to focus on lower (or upper) expectations, and that \emph{lower} (and \emph{upper}) \emph{probabilities} can be regarded as a special case; for example, for any $A\subseteq\states$, we have that $\underline{P}_{\rateset,\mathcal{M}}(X_s\in A\,\vert\,X_t) \coloneqq \inf\{P(X_s\in A\vert X_t)\,:\,P\in\mathbb{P}_{\rateset,\mathcal{M}}\}=\lexp[\ind{A}(X_s)\,\vert\,X_t]$, where $\ind{A}$ is the indicator of $A$, defined for all $x\in\states$ by $\ind{A}(x)\coloneqq1$ if $x\in A$ and $\ind{A}(x)\coloneqq0$ otherwise.

In the sequel, we will assume that $\mathcal{M}$ is non-empty, and that $\rateset$ is non-empty, bounded,\footnote{That is, that there exists a $c\in\realsnonneg$ such that, for all $Q\in\rateset$ and $x\in\states$, it holds that $\abs{Q(x,x)}<c$.} convex, and has \emph{separately specified rows}. This latter property states that $\rateset$ is closed under arbitrary recombination of rows from its elements; see~\citep[Definition 24]{krak2016ictmc} for a formal definition. 
Under these assumptions, $\mathbb{P}_{\rateset,\mathcal{M}}$ satisfies an \emph{imprecise Markov property}, in the sense that $\lexp[f(X_s)\,\vert\,X_t,X_u=x_u]=\lexp[f(X_s)\,\vert\,X_t]$. This property explains why we call this model an imprecise continuous-time ``Markov'' chain.

\subsection{Computing Lower Expectations for ICTMC's}\label{subsec:ICTMC_computations}

Because we want to focus in this paper on providing efficient methods of computation, we here briefly recall some previous results from~\citet{krak2016ictmc} about how to compute lower expectations for ICTMC's. We focus in particular on how to do this for functions on a single time-point. 

To this end, it is useful to introduce the \emph{lower transition rate operator} $\lrate$ that corresponds to $\rateset$. This operator is a map from $\gamblesX$ to $\gamblesX$, defined for every $f\in\gamblesX$ by
\begin{equation}\label{eq:lower_rate_is_inf}
\left[\,\lrate f\right](x) \coloneqq \inf\left\{ \sum_{x'\in\states}Q(x,x')f(x')\,:\, Q\in\rateset \right\}
~~\text{for all $x\in\states$}.
\end{equation}

Using this lower transition rate operator $\lrate$, we can compute conditional lower expectations in the following way. For any $t,s\in\realsnonneg$, with $t\leq s$, and any $f\in\gamblesX$, it has been shown that
\begin{equation*}
\lexp[f(X_s)\,\vert\,X_t] = \underline{\mathbb{E}}_\rateset[f(X_s)\,\vert\,X_t] \coloneqq \lim_{n\to+\infty}\left[I+\frac{(s-t)}{n}\lrate\right]^n f\,,
\end{equation*}
where $I$ is the identity operator on $\gamblesX$, in the sense that $I g=g$ for every $g\in\gamblesX$.
The notation $\underline{\mathbb{E}}_\rateset$ is meant to indicate that this conditional lower expectation only depends on $\rateset$, and not on $\mathcal{M}$. The above implies that for large enough $n\in\nats$, and writing $\Delta\coloneqq \nicefrac{(s-t)}{n}$, we have
\begin{equation}\label{eq:lower_exp_in_steps}
\lexp[f(X_s)\,\vert\,X_t] = \underline{\mathbb{E}}_\rateset[f(X_s)\,\vert\,X_t] \approx \bigl[I + \Delta\lrate\,\bigr]^nf\,.
\end{equation}
Concretely, this means that if one is able to solve the minimisation problem in Equation~\eqref{eq:lower_rate_is_inf}---which is relatively straightforward for ``nice enough'' $\rateset$, e.g., convex hulls of finite sets of rate matrices---then one can also compute conditional lower expectations using the expression in Equation~\ref{eq:lower_exp_in_steps}. In practice, we do this by first computing $f_1'\coloneqq \lrate f$ using Equation~\eqref{eq:lower_rate_is_inf}, and then computing $f_1\coloneqq f + \Delta f_1'$. Next, we compute $f_2'\coloneqq \lrate f_1$, from which we obtain $f_2\coloneqq f_1 + \Delta f_2'$. Proceeding in this fashion, after $n$ steps we then finally obtain $f_n \coloneqq [I+\Delta\lrate]f_{n-1} = \bigl[I+\Delta\lrate\bigr]^nf$, which is roughly the quantity of interest $\underline{\mathbb{E}}_{\rateset,\mathcal{M}}[f(X_s)\,\vert\,X_t]$ provided that $n$ was taken large enough.\footnote{We refer the reader to~\citep[Proposition 8.5]{krak2016ictmc} for a theoretical bound on the minimum such $n$ that is required to ensure a given maximum error on the approximation in Equation~\eqref{eq:lower_exp_in_steps}. We here briefly note that this bound scales polynomially in every relevant parameter. This means that $\lexp[f(X_s)\,\vert\,X_t]$ is numerically computable in polynomial time, provided that $\rateset$ is such that Equation~\eqref{eq:lower_rate_is_inf} can also be solved in the same time-complexity order.}

As noted above, the conditional lower expectation $\lexp[f(X_s)\vert X_t]$ only depends on $\rateset$. Similarly, and in contrast, the unconditional lower expectation at time zero only depends on $\mathcal{M}$. That is,
\begin{equation}\label{eq:unconditional_time_zero}
\lexp[f(X_0)] = \underline{\mathbb{E}}_{\mathcal{M}}[f(X_0)] \coloneqq \inf\left\{ \sum_{x\in\states}p(x)f(x)\,:\,p\in\mathcal{M} \right\}\,.
\end{equation}
Furthermore, the unconditional lower expectation at an arbitrary time $t\in\realsnonneg$, is given by
\begin{equation}\label{eq:unconditional_lower_exp}
\underline{\mathbb{E}}_{\rateset,\mathcal{M}}[f(X_t)] = \underline{\mathbb{E}}_{\mathcal{M}}\bigl[\underline{\mathbb{E}}_{\rateset}[f(X_t)\,\vert\,X_0]\bigr]\,,
\end{equation}
which can therefore be computed by combining Equations~\eqref{eq:lower_exp_in_steps} and~\eqref{eq:unconditional_time_zero}. In particular, from a practical point of view, it suffices to first compute the conditional lower expectation $\underline{\mathbb{E}}_{\rateset}[f(X_t)\,\vert\,X_0]$, using Equation~\eqref{eq:lower_exp_in_steps}. Once this quantity is obtained, it remains to compute the right-hand side of Equation~\eqref{eq:unconditional_time_zero}, which again is relatively straightforward when $\mathcal{M}$ is ``nice enough'', e.g., the convex hull of some finite set of probability mass functions.

\section{Imprecise Continuous-Time Hidden Markov Chains}\label{sec:icthmc}

In this section, we construct the \emph{hidden} model that is the subject of this paper. Our aim is to augment the stochastic processes that were introduced in the previous section, by adding random \emph{output} variables $Y_t$ whose distribution depends on the state $X_t$ at the same time point $t$.

We want to focus in this paper on the more practical aspect of solving the inference problem of interest, i.e., computing lower expectations on the state-space \emph{given some observations}.
Hence, we will assume that we are given some finite sequence of time points, and we then only consider these time points in augmenting the model. 
In order to disambiguate the notation, we will henceforth denote stochastic processes as $P_\states$, to emphasise that they are only concerned with the state-space.
\vspace{-4pt}

\subsection{Output Variables}\label{sec:observs}

We want to augment stochastic processes with random ``output variables'' $Y_t$, whose distribution depends on the state $X_t$. We here define the corresponding (conditional) distribution.

We want this definition to be fairly general, and in particular do not want to stipulate that $Y_t$ should be either a discrete or a continuous random variable. To this end, we simply consider some set $\observs$ to be the outcome space of the random variable. We then let $\Sigma$ be some algebra on $\observs$. Finally, for each $x\in\states$, we consider some finitely (and possibly $\sigma$-)additive probability measure $P_{\observs\vert\states}(\cdot\vert x)$ on $(\observs,\Sigma)$, with respect to which the random variable $Y_t$ can be defined.
%\vspace{-4pt}

\begin{definition}
An \emph{output model} is a tuple $(\observs,\Sigma,P_{\observs\vert \states})$, where $\observs$ is an outcome space, $\Sigma$ is an algebra on $\observs$, and, for all $x\in\states$, $P_{\observs\vert\states}(\cdot\vert x)$ is a finitely additive probability measure on $(\observs,\Sigma)$.
\end{definition}
%\vspace{-4pt}

When considering (multiple) explicit time points, we use notation analogous to that used for states; so, $\observs_t\coloneqq\observs$ for any time $t\in\realsnonneg$, and for any $u\in\mathcal{U}$, we write $\observs_u\coloneqq \prod_{t\in u}\observs_{t}$. 

We let $\Sigma_u$ denote the set of all events of the type $O_u=\times_{t\in u}O_t$, where, for all $t\in u$, $O_{t}\in\Sigma$. 
This set $\Sigma_u$ lets us describe observations using assessments of the form $(Y_t\in O_t \text{~for all $t\in u$})$.   
For any $O_u\in\Sigma_u$ and $x_u\in\states_u$, we also adopt the shorthand notation $P_{\observs\vert\states}(O_u\vert x_u)\coloneqq \prod_{t\in u}P_{\observs\vert\states}(O_t\vert x_t)$.
\vspace{-4pt}

\subsection{Augmented Stochastic Processes}\label{sec:aug_stochastic_processes}
We now use this notion of an output model to define the stochastic model $P$ that corresponds to a---precise---continuous-time \emph{hidden} stochastic process. 
So, consider some fixed output model $(\observs,\Sigma,P_{\observs\vert\states})$, some fixed continuous-time stochastic process $P_\states$ and some fixed, non-empty and finite sequence of time-points $u\in\mathcal{U}$ on which observations of the outputs may take place. 

We assume that $Y_t$ is conditionally independent of \emph{all} other variables, given the state $X_t$. This means that the construction of the augmented process $P$ is relatively straightforward; we can simply multiply $P_{\observs\vert\states}(\cdot\,\vert\,X_t)$ with any distribution $P_\states(X_t,\cdot)$ that includes $X_t$ to obtain the joint distribution including $Y_t$: for any $t\in u$ and $v\in\mathcal{U}$ such that $t\notin v$, any $x_t\in\states_t$ and $x_v\in\states_v$, and any $O_t\in\Sigma$:
\vspace{-2pt}
\begin{equation*}
P(Y_t\in O_t,X_t=x_t,X_v=x_v) \coloneqq P_{\observs\vert\states}(O_t\,\vert\,x_t)P_\states(X_t=x_t,X_v=x_v)\,.\vspace{-2pt}
\end{equation*}
Similarly, when considering multiple output observations at once---say for the entire sequence $u$---then for any $v\in\mathcal{U}$ such that $u\cap v=\emptyset$, any $x_u\in\states_u$ and $x_v\in\states_v$, and any $O_u\in\Sigma_u$:
\vspace{-2pt}
\begin{equation*}
P(Y_u\in O_u,X_u=x_u, X_v=x_v) \coloneqq P_{\observs\vert\states}(O_{u}\,\vert\,x_{u})P_\states(X_u=x_u,X_v=x_v)\,.\vspace{-2pt}
\end{equation*}
Other probabilities can be derived by appropriate marginalisation.
We denote the resulting augmented stochastic process as $P=P_{\observs\vert\states}\otimes P_\states$,
for the specific output model $P_{\observs\vert\states}$ and stochastic process $P_\states$ that were taken to be fixed in this section.
\vspace{-8pt}

\subsection{Imprecise Continuous-Time Hidden Markov Chains}\label{subsec:ICTHMC}

An \emph{imprecise continuous-time hidden Markov chain} (ICTHMC) is a set of augmented stochastic processes, obtained by augmenting all processes in an ICTMC with some given output model.
\begin{definition}\label{def:hidden_ictmc}
Consider any ICTMC $\mathbb{P}_{\rateset,\mathcal{M}}$, and any output model $(\observs,\Sigma,P_{\observs\vert\states})$. Then, the corresponding \emph{imprecise continuous-time hidden Markov chain} (ICTHMC) $\mathcal{Z}$ is the set of augmented stochastic processes that is defined by
$\mathcal{Z} \coloneqq \left\{ P_{\observs\vert\states}\otimes P_{\states} \,:\, P_{\states}\in\mathbb{P}_{\rateset,\mathcal{M}}\right\}$.
The lower expectation with respect to $\mathcal{Z}$ will be denoted by $\underline{\mathbb{E}}_\mathcal{Z}$.
\end{definition}
Note that we leave the parameters $\mathcal{M}$, $\rateset$ and $P_{\observs\vert\states}$ implicit in the notation of the ICTHMC $\mathcal{Z}$---we will henceforth take these parameters to be fixed.

Also, the output model is taken to be precise, and shared by all processes in the set. One further generalisation that we aim to make in the future is to allow for an imprecise specification of this output model. However, this would force us into choosing an appropriate notion of independence; e.g., whether to enforce the independence assumptions made in Section~\ref{sec:aug_stochastic_processes}, leading to strong or complete independence, or to only enforce the lower envelopes to have these independence properties, leading to epistemic irrelevance. It is currently unclear which choice should be preferred, e.g. with regard to computability, so at present we prefer to focus on this simpler model.
\vspace{-8pt}

\section{Updating the Model}\label{sec:updating_model}

%In the context of \emph{hidden} (continuous-time) Markov chains, it is typically assumed that the state $X_t$ that is obtained by the process at time $t$, cannot be directly observed---hence the term ``hidden''. Rather, we can only observe realisations of the output variable $Y_t$. 

Suppose now that we have observed that some event $(Y_u\in O_u)$ has taken place, with $O_u\in\Sigma_u$. We here use the terminology that we \emph{update} our model with these observations, after which the updated model reflects our revised beliefs about some quantity of interest. These updated beliefs, about some function $f\in\gambles(\states_v)$, say, are then denoted by
$\mathbb{E}_P[f(X_v)\,\vert\,Y_u\in O_u]$
or $\underline{\mathbb{E}}_\mathcal{Z}[f(X_v)\,\vert\,Y_u\in O_u]$, 
depending on whether we are considering a precise or an imprecise model. In this section, we provide definitions and alternative expressions for such updated (lower) expectations.
\vspace{-8pt}

\subsection{Observations with Positive (Upper) Probability}\label{subsec:pos_prob}

When our assertion $(Y_u\in O_u)$ about an observation at time points $u$ has positive probability, we can---in the precise case---update our model by application of Bayes' rule. The following gives a convenient expression for the updated expectation $\mathbb{E}_P[f(X_v)\,\vert\,Y_u\in O_u]$, which makes use of the independence assumptions in Section~\ref{sec:aug_stochastic_processes} for augmented stochastic processes.
%\vspace{-4pt}
\begin{proposition}\label{prop:precise_conditioning_for_positive}
Let $P$ be an augmented stochastic process and consider any $u,v\in\mathcal{U}$, $O_u\in\Sigma_u$ and $f\in\gambles(\states_v)$. Then the updated expectation is given by
\begin{equation*}
\mathbb{E}_P[f(X_v)\,\vert\,Y_u\in O_u] \coloneqq \sum_{x_v\in\states_v}f(x_v)\frac{P(X_v=x_v, Y_u\in O_u)}{P(Y_u\in O_u)} = \frac{\mathbb{E}_{P_\states}[f(X_v)P_{\observs\vert\states}(O_u\vert X_u)]}{\mathbb{E}_{P_\states}[P_{\observs\vert\states}(O_u\,\vert\,X_u)]}\,,
\end{equation*}
whenever $P(Y_u\in O_u)=\mathbb{E}_{P_\states}[P_{\observs\vert\states}(O_u\,\vert\,X_u)]>0$, and is left undefined, otherwise.
\end{proposition}

Having defined above how to update all the precise models $P\in\mathcal{Z}$, we will now update the imprecise model through \emph{regular extension}~\citep{Walley:1991vk}. This corresponds to simply discarding from $\mathcal{Z}$ those precise models that assign zero probability to $(Y_u\in O_u)$, updating the remaining models, and then computing their lower envelope.

\begin{definition}\label{def:reg_ext_pos}
Let $\mathcal{Z}$ be an ICTHMC and consider any $u,v\in\mathcal{U}$, $O_u\in\Sigma_u$ and $f\in\gambles(\states_v)$. Then the updated lower expectation is defined by
\begin{equation*}
\underline{\mathbb{E}}_{\mathcal{Z}}\bigl[f(X_v)\,\vert\,Y_u\in O_u\bigr] \coloneqq \inf\bigl\{ \mathbb{E}_P[f(X_v)\,\vert\,Y_u\in O_u]\,:\, P\in\mathcal{Z},\, P(Y_u\in O_u)>0 \bigr\}\,,
\end{equation*}
whenever $\overline{P}_\mathcal{Z}(Y_u\in O_u)=\uexp[P_{\observs\vert\states}(O_u\,\vert\,X_u)]>0$, and is left undefined, otherwise.
\end{definition}

As is well known, the updated lower expectation that is obtained through regular extension satisfies Walley's \emph{generalised Bayes' rule}~\citep{Walley:1991vk}. The following proposition gives an expression for this generalised Bayes' rule, rewritten using some of the independence properties of the model. We will shortly see why this expression is useful from a computational perspective.
\begin{proposition}\label{prop:GBR_regular}
Let $\mathcal{Z}$ be an ICTHMC and consider any $u,v\in\mathcal{U}$, $O_u\in\Sigma_u$ and $f\in\gambles(\states_v)$. Then, if $\overline{P}_\mathcal{Z}(Y_u\in O_u) = \uexp[P_{\observs\vert\states}(O_u\,\vert\,X_u)] > 0$, the quantity $\underline{\mathbb{E}}_{\mathcal{Z}}\bigl[f(X_v)\,\vert\,Y_u\in O_u\bigr]$ satisfies
\begin{equation*}
\underline{\mathbb{E}}_{\mathcal{Z}}\bigl[f(X_v)\,\vert\,Y_u\in O_u\bigr] = \max\left\{\mu\in\reals\,:\, \lexp\bigl[P_{\observs\vert\states}(O_u\vert X_u)\bigl(f(X_v) - \mu\bigr)\bigr] \geq 0\right\}\,.
\end{equation*}
\end{proposition}

\subsection{Uncountable Outcome Spaces, Point Observations, and Probability Zero}\label{subsec:uncountable}

An important special case where observations have probability zero for all precise models, but where we can still make informative inferences, is when we have an uncountable outcome space $\observs$ and the observations are points $y_u\in\observs_u$---i.e., when $Y_u$ is continuous. In this case, it is common practice to define the updated expectation $\mathbb{E}_P[f(X_v)\,\vert\,Y_u=y_u]$ as a limit of \emph{conditional} expectations, where each conditioning event is an increasingly smaller region around this point $y_u$. We will start by formalising this idea in a relatively abstract way, but will shortly make this practicable. For the sake of intuition, note that we are working towards the introduction of probability density functions.

Fix any $P\in\mathcal{Z}$, consider any $y_u\in\observs_u$ and choose a sequence $\{O_u^i\}_{i\in\nats}$ of events in $\Sigma_u$ which shrink to $y_u$---i.e., such that $O_u^i\supseteq O_u^{i+1}$ for all $i\in\nats$, and such that $\cap_{i\in\nats} O_u^i=\{y_u\}$. We then define
\begin{equation}\label{eq:def:precise_updated_limit}
\mathbb{E}_P[f(X_v)\,\vert\,Y_u=y_u] \coloneqq \lim_{i\to+\infty} \mathbb{E}_P[f(X_v)\,\vert\,Y_u\in O_u^i]\,.
\end{equation}
This limit exists if there is a sequence $\{\lambda_i\}_{i\in\nats}$ in $\realspos$ such that, for every $x_u\in\states_u$, the limit
\begin{equation*}
\phi_u(y_u\,\vert\, x_u) \coloneqq \lim_{i\to+\infty}\frac{P_{\observs\vert\states}(O_u^i\,\vert\, x_u)}{\lambda_i}
\end{equation*}
exists, is real-valued---in particular, finite---and satisfies $\mathbb{E}_{P_\states}[\phi_u(y_u\,\vert\,X_u)]>0$:
\begin{proposition}\label{prop:precise_bayes_rule_densities}
Let $P$ be an augmented stochastic process and consider any $u,v\in\mathcal{U}$, $y_u\in\observs_u$ and $f\in\gambles(\states_v)$. For any $\{O_u^i\}_{i\in\nats}$ in $\Sigma_u$ that shrinks to $y_u$, if for some $\{\lambda_i\}_{i\in\nats}$ in $\realspos$ the quantity $\phi_u(y_u\,\vert\,X_u)$ exists, is real-valued, and satisfies $\mathbb{E}_{P_\states}[\phi_u(y_u\,\vert\,X_u)]>0$, then
\begin{equation}\label{eq:updated_expectation_is_limit}
\mathbb{E}_P[f(X_v)\,\vert\,Y_u=y_u] \coloneqq \lim_{i\to+\infty} \mathbb{E}_P[f(X_v)\,\vert\,Y_u\in O_u^i] = \frac{\mathbb{E}_{P_\states}[f(X_v)\phi_u(y_u\vert X_u)]}{\mathbb{E}_{P_\states}[\phi_u(y_u\,\vert\,X_u)]}\,.
\end{equation}
\end{proposition}
Note that $\mathbb{E}_P[f(X_v)\,\vert\,Y_u=y_u]$ is clearly dependent on the exact sequence $\{O_u^i\}_{i\in\nats}$. Unfortunately, this is the best we can hope for at the level of generality that we are currently dealing with. 
For brevity, we nevertheless omit from the notation the updated expectation's dependency on this sequence. However, as we will explain below, this should not be problematic for most practical applications.

It is also useful to note that $\phi_u(y_u\vert x_u)$ can often be constructed ``piecewise''. That is, if for every $t\in u$ there is a sequence $\{\lambda_{\,t,i}\}_{i\in\nats}$ in $\realspos$ such that, for all $x_t\in\states_t$,
\begin{equation*}
\phi_t(y_t\vert x_t)\coloneqq \lim_{i\to+\infty} \frac{P_{\observs\vert\states}(O_t^i\vert x_t)}{\lambda_{\,t,i}}
\end{equation*}
exists and is real-valued, then choosing $\{\lambda_i\}_{i\in\nats}$ as $\lambda_i\coloneqq \prod_{t\in u}\lambda_{\,t,i}$ yields $\phi_u(y_u\vert x_u)=\prod_{t\in u}\phi_t(y_t\vert x_t)$. 
%The converse often also holds, in that if $\phi_u(y_u\vert x_u)$ exists there tend to be sequences $\{\lambda_{\,t,i}\}_{t\in\nats}$ so that $\phi_t(y_t\vert x_t)$ exists, and for which $\prod_{t\in u}\phi_t(y_t\vert x_t)=\phi_u(y_u\vert x_u)$. The ``often'' qualifier is a bit hard to make rigorous at this level of generality, but, in any case, this direction is less useful practically. 

Now, to make the above practicable, we can for example assume that if $\observs$ is uncountable, then it is the set $\observs=\reals^d$, for some $d\in\nats$, and that $\Sigma$ is the Borel $\sigma$-algebra on $\reals^d$. 
For each $x\in\states$, we then assume that the measure $P_{\observs\vert\states}(\cdot\,\vert x)$ is induced by some given \emph{probability density function}: a measurable function $\psi(\cdot\,\vert x):\observs\to\realsnonneg$ such that $\int_\observs \psi(y\vert x)\,\mathrm{d}y=1$ and, for every $O\in\Sigma$,
\begin{equation*}
P_{\observs\vert\states}(O\,\vert x) \coloneqq \int_O \psi(y\vert x)\,\mathrm{d}y\,,
\end{equation*}
where the integrals are understood in the Lebesgue sense.

Then choose any $y_u\in\observs_u$, any $t\in u$, any sequence $\{O_t^i\}_{i\in\nats}$ of open balls in $\observs_t$ that are centred on, and shrink to, $y_t$, and fix any $x_u\in\states_u$. If $\psi(\cdot\vert x_t)$ is continuous at $y_t$, it can be shown that
\begin{equation}\label{eq:density_is_limit}
\phi_t(y_t\vert x_t) = \lim_{i\to+\infty} \frac{P_{\observs\vert\states}(O_t^i\vert x_t)}{\lambda(O_t^i)} = \psi(y_t\vert x_t)\,,
\end{equation}
where $\lambda(O_t^i)$ denotes the Lebesgue measure of $O_t^i$. So, we can construct the sequence $\{O_u^i\}_{i\in\nats}$ such that every $O_u^i\coloneqq \prod_{t\in u}O_t^i$, with each $O_t^i$ chosen as above. If we then choose the sequence $\{\lambda_i\}_{i\in\nats}$ as $\lambda_i\coloneqq \prod_{t\in u}\lambda(O_t^i)$ for each $i\in\nats$, we find 
$\phi_u(y_u\vert x_u) = \prod_{t\in u}\phi_t(y_t\vert x_t)=\prod_{t\in u}\psi(y_t\vert x_t)$, 
provided that each $\phi_t(y_t\vert x_t)$ satisfies Equation~\eqref{eq:density_is_limit}. 
It can now be seen that, under these assumptions, the right-hand side of Equation~\eqref{eq:updated_expectation_is_limit} is simply the well-known Bayes' rule for (finite) mixtures of densities. 

In most practical applications, therefore, the function $\phi_u(\cdot\,\vert\,x_u)$ is known explicitly; one may assume, for example, that $Y_t$ follows a Normal distribution with parameters depending on $X_t$, and the functions $\phi_t(\cdot\,\vert\,x_t)$---and by extension, $\phi_u(\cdot\vert x_u)$---then follow directly by identification with $\psi(\cdot\,\vert x_t)$. Furthermore, arguably, most of the density functions that one encounters in practice will be continuous and strictly positive at $y_t$. This guarantees that the limit in Equation~\eqref{eq:density_is_limit} exists, and largely solves the interpretation issue mentioned above: when $\phi_u(y_u\vert X_u)=\prod_{t\in u}\psi(y_t\vert X_t)$ is continuous and positive at $y_u$, $\mathbb{E}_P[f(X_v)\,\vert\,Y_u=y_u]$ exists and is the same for almost\footnote{\label{fnote:limit_required_properties}
It suffices if, for all $t\in u$, there is a sequence of open balls $\{B_t^i\}_{i\in\nats}$ in $\observs$ that shrinks to $y_t$ such that, for all $i\in\nats$, $O_t^i$ has positive Lebesgue measure and is contained in $B_t^i$.} all sequences $\{O_u^i\}_{i\in\nats}$.

Moving on, note that if $\phi_u(y_u\vert X_u)$ exists and satisfies $\lexp[\phi_u(y_u\vert X_u)]>0$, then the updated expectation $\mathbb{E}_P[f(X_v)\,\vert\,Y_u=y_u]$ is well-defined for every $P\in\mathcal{Z}$. Hence, we can then update the imprecise model by updating each of the precise models that it consists of.
\begin{definition}\label{def:reg_ext_densities}
Let $\mathcal{Z}$ be an ICTHMC and consider any $u,v\in\mathcal{U}$, $y_u\in\observs_u$, and $f\in\gambles(\states_v)$. For any $\{O_u^i\}_{i\in\nats}$ in $\Sigma_u$ that shrinks to $y_u$, if for some $\{\lambda_i\}_{i\in\nats}$ in $\realspos$ the quantity $\phi_u(y_u\,\vert\,X_u)$ exists and is real-valued, the updated lower expectation is defined by
\begin{equation*}
\underline{\mathbb{E}}_{\mathcal{Z}}\bigl[f(X_v)\,\vert\,Y_u = y_u\bigr] \coloneqq \inf\{\mathbb{E}_P[f(X_v)\vert Y_u=y_u]\,:\,P\in\mathcal{Z}\}\,,
\end{equation*}
whenever $\lexp[\phi_u(y_u\vert X_u)] >0$, and is left undefined, otherwise.
\end{definition}

Similar to the results in Section~\ref{subsec:pos_prob}, this updated lower expectation satisfies a ``generalised Bayes' rule for mixtures of densities'', in the following sense.

\begin{proposition}\label{prop:GBR_for_densities_lower_zero}
Let $\mathcal{Z}$ be an ICTHMC and consider any $u,v\in\mathcal{U}$, $y_u\in\observs_u$ and $f\in\gambles(\states_v)$. For any $\{O_u^i\}_{i\in\nats}$ in $\Sigma_u$ that shrinks to $y_u$, if for some $\{\lambda_i\}_{i\in\nats}$ in $\realspos$ the quantity $\phi_u(y_u\,\vert\,X_u)$ exists, is real-valued, and satisfies $\lexp[\phi_u(y_u\vert X_u)]>0$, then
\begin{equation}\label{eq:gbr_densities}
\underline{\mathbb{E}}_{\mathcal{Z}}\bigl[f(X_v)\,\vert\,Y_u = y_u\bigr] = \max\left\{\mu\in\reals\,:\, \lexp\bigl[\phi_u(y_u\vert X_u)\bigl(f(X_v) - \mu\bigr)\bigr] \geq 0\right\}\,.
\end{equation}
\end{proposition}
Furthermore, this updated imprecise model can be given an intuitive limit interpretation.
\begin{proposition}\label{prop:GBR_for_densities_is_limit_if_continuous}
Let $\mathcal{Z}$ be an ICTHMC and consider any $u,v\in\mathcal{U}$, $y_u\in\observs_u$ and $f\in\gambles(\states_v)$. For any $\{O_u^i\}_{i\in\nats}$ in $\Sigma_u$ that shrinks to $y_u$, if for some $\{\lambda_i\}_{i\in\nats}$ in $\realspos$ the quantity $\phi_u(y_u\,\vert\,X_u)$ exists, is real-valued, and satisfies $\lexp[\phi_u(y_u\vert X_u)]>0$, then $\underline{\mathbb{E}}_\mathcal{Z}[f(X_v)\vert Y_u=y_u] 
 = \lim_{i\to+\infty}\underline{\mathbb{E}}_\mathcal{Z}[f(X_v)\vert Y_u\in O_u^i]$.
\end{proposition}

Now, recall that the requirement $\lexp[\phi_u(y_u\vert X_u)]>0$ for updating the imprecise model is a sufficient condition to guarantee that \emph{all} the precise updated models are well-defined. However, one may wonder whether it is also possible to update the imprecise model under weaker conditions. Indeed, one obvious idea would be to define the updated model more generally as
\begin{equation*}
\underline{\mathbb{E}}_\mathcal{Z}^\mathrm{R}[f(X_v)\,\vert\,Y_u=y_u] \coloneqq \inf\left\{ \mathbb{E}_P[f(X_v)\,\vert\,Y_u=y_u]\,:\, P\in\mathcal{Z},\,\mathbb{E}_{P_\states}[\phi_u(y_u\vert X_u)]>0 \right\}\,,
\end{equation*}
whenever $\uexp[\phi_u(y_u\vert X_u)]>0$; this guarantees that \emph{some} of the precise updated models are well-defined. This updated lower expectation satisfies the same generalised Bayes' rule as above, i.e. the right-hand side of Equation~\eqref{eq:gbr_densities} is equal to $\underline{\mathbb{E}}_\mathcal{Z}^\mathrm{R}[f(X_v)\,\vert\,Y_u=y_u]$ whenever $\uexp[\phi_u(y_u\vert X_u)]>0$. However, the limit interpretation then fails to hold, in the sense that it is possible to construct an example where $\underline{\mathbb{E}}_\mathcal{Z}^\mathrm{R}[f(X_v)\,\vert\,Y_u=y_u] \neq \lim_{i\to+\infty} \underline{\mathbb{E}}_\mathcal{Z}[f(X_v)\,\vert\,Y_u\in O_u^i]$, with $\uexp[\phi_u(y_u\vert X_u)]>0$ but $\lexp[\phi_u(y_u\vert X_u)]=0$. We feel that this makes this more general updating scheme somewhat troublesome from an interpretation (and hence philosophical) point of view.

On the other hand, we recall that the existence of $\phi_u(y_u\vert X_u)$ and the positivity of $\mathbb{E}_{P_\states}[\phi_u(y_u\vert X_u)]$ are necessary and sufficient conditions for the limit in Equation~\eqref{eq:def:precise_updated_limit} to exist and be computable using Equation~\eqref{eq:updated_expectation_is_limit}. However, these conditions are sufficient but non-necessary for that limit to simply exist. Therefore, a different way to generalise the imprecise updating method would be
\begin{equation*}
\underline{\mathbb{E}}_\mathcal{Z}^\mathrm{L}[f(X_v)\,\vert\,Y_u=y_u] \coloneqq \inf\left\{ \mathbb{E}_P[f(X_v)\,\vert\,Y_u=y_u]\,:\, P\in\mathcal{Z},~\text{$\mathbb{E}_P[f(X_v)\,\vert\,Y_u=y_u]$ exists} \right\}\,,
\end{equation*}
whenever $\{P\in\mathcal{Z}\,:\,~\text{$\mathbb{E}_P[f(X_v)\,\vert\,Y_u=y_u]$ exists}\}\neq\emptyset$. We conjecture that this updated model \emph{does} satisfy the limit interpretation, but on the other hand, it is possible to show that this, in turn, no longer satisfies the above generalised Bayes' rule. That makes this updating scheme somewhat troublesome from a practical point of view because, as we discuss below, the expression in Equation~\eqref{eq:gbr_densities} is crucial for our method of efficient computation of the updated lower expectation.

\section{Inference Algorithms}\label{sec:inference_algos}

In the previous section, we have seen that we can use the generalised Bayes' rule for updating our ICTHMC with some given observations. From a computational point of view, this is particularly useful because, rather than having to solve the non-linear optimisation problems 
in Definitions~\ref{def:reg_ext_pos} or~\ref{def:reg_ext_densities} directly, 
we can focus on evaluating the function $\lexp\bigl[P_{\observs\vert\states}(O_u\vert X_u)\bigl(f(X_v) - \mu\bigr)\bigr]$,
or its density-analogue, for some fixed value of $\mu$. Finding the updated lower expectation is then a matter of finding the maximum value of $\mu$ for which this quantity is non-negative. As we will discuss in Section~\ref{sec:gbr}, this is a relatively straightforward problem to solve numerically.

Therefore, in order for this approach to be computationally tractable, we require efficient algorithms that can evaluate this quantity for a given value of $\mu$. In Section~\ref{sec:funcs_single_time}, we provide such an algorithm for the important case where the function $f$ depends on a single time-point.

We first generalise the problem so that these results are applicable both for observations of the form $(Y_u\in O_u)$, and for point-observations $(Y_u=y_u)$ in an uncountable outcome space. Recall that 
\begin{equation*}
P_{\observs\vert\states}(O_u\vert X_u) = \prod_{t\in u}P_{\observs\vert\states}(O_{t}\vert X_{t})\,\quad\quad\text{and}\quad\quad \phi_u(y_u\vert X_u) = \prod_{t\in u}\phi_{t}(y_{t}\vert X_{t})\,.\vspace{-5pt}
\end{equation*}
In both cases, we can rewrite this expression as $\prod_{t\in u}g_{t}(X_{t})$, where, for all $t\in u$, $g_{t}\in\gambles(\states_{t})$ and $g_{t}\geq 0$. The function of interest is then
$\lexp\left[ \bigl(\prod_{t\in u}g_{t}(X_{t})\bigr)\bigl(f(X_v) - \mu\bigr) \right]$ and the sign conditions in Propositions~\ref{prop:GBR_regular} and~\ref{prop:GBR_for_densities_lower_zero} reduce to $\uexp[\prod_{t\in u} g_{t}(X_{t})]>0$ and $\lexp[\prod_{t\in u} g_{t}(X_{t})]>0$, respectively.

\subsection{Solving the Generalised Bayes' Rule}\label{sec:gbr}

Finding the maximum value of $\mu$ for which the function of interest in the generalised Bayes' rule is non-negative, is relatively straightforward numerically. This is because this function, parameterised in $\mu$, is very well-behaved. The proposition below explicitly states some of its properties. These are essentially well-known, and can also be found in other work; see, e.g.,~\cite[Section 2.7.3]{de2015credal}. The statement below is therefore intended to briefly recall these properties, and is stated in a general form where we can also use it when working with densities.

\begin{proposition}\label{prop:GBR_properties}
Let $\mathbb{P}_{\rateset,\mathcal{M}}$ be an ICTMC and consider any $u,v\in\mathcal{U}$, any $f\in\gambles(\states_v)$ and, for all $t\in u$, any $g_{t}\in\gambles(\states_{t})$ such that $g_{t}\geq 0$. 
Consider the function $G: \reals\to\reals$ that is given, for all $\mu\in\reals$, by $G(\mu)\coloneqq \lexp\left[\left(\prod_{t\in u} g_{t}(X_{t})\right)\bigl(f(X_v) - \mu\bigr)\right]$.
Then the following properties hold: \vspace{-2pt}
\begin{enumerate}[label=G\arabic*:,ref=G\arabic*]
\item $G$ is continuous, non-increasing, concave, and has a root, i.e. $\exists \mu\in\reals:G(\mu)=0$. \label{GBR:always} \vspace{-3pt}
\item If\/ $\lexp\bigl[\prod_{t\in u} g_{t}(X_{t})\bigr] >0$, then $G$ is (strictly) decreasing, and has a unique root. \label{GBR:low_pos} \vspace{-3pt}
\item If\/ $\lexp\bigl[\prod_{t\in u} g_{t}(X_{t})\bigr]=0$ but $\uexp\bigl[\prod_{t\in u} g_{t}(X_{t})\bigr] >0$, then $G$ has a maximum root $\mu_*$, satisfies $G(\mu)=0$ for all $\mu\leq \mu_*$, and is (strictly) decreasing for $\mu>\mu_*$. \label{GBR:up_pos} \vspace{-3pt}
\item If\/ $\uexp\bigl[\prod_{t\in u} g_{t}(X_{t})\bigr]=0$, then $G$ is identically zero, i.e. $\forall \mu\in\reals: G(\mu)=0$. \label{GBR:none_pos}
\end{enumerate}
\end{proposition}

Note that the function $G$ in Proposition~\ref{prop:GBR_properties} can behave in three essentially different ways. These correspond to the cases where the observed event has strictly positive probability(/density) for \emph{all} processes in the set; to where it only has positive probability(/density) for \emph{some} processes; and to where it has \emph{zero} probability(/density) for \emph{all} processes.
In the first two cases---which are the important ones to apply the generalised Bayes' rule---the function is ``well-behaved'' enough to make finding its maximum root a fairly simple task. For instance, a standard bisection/bracketing algorithm can be applied here, known in this context as Lavine's algorithm~\citep{cozman1997alternatives}.

We sketch this method below. First, note that due to Propositions~\ref{prop:GBR_regular} and~\ref{prop:GBR_for_densities_lower_zero}, the maximum root will always be found in the interval $[\min f, \max f]$. The properties above therefore provide us with a way to check the sign conditions for updating. That is, for any $\mu>\max f$, we see that $G(\mu)<0$ if and only if $\uexp[\prod_{t\in u} g_{t}(X_{t})]>0$;  
similarly, for any $\mu < \min f$, we see that $G(\mu)>0$ if and only if $\lexp[\prod_{t\in u} g_{t}(X_{t})]>0$. Evaluating $G$ at such values of $\mu$ is therefore sufficient to check the sign conditions in Propositions~\ref{prop:GBR_regular} and~\ref{prop:GBR_for_densities_lower_zero}.

The algorithm now starts by setting $\mu_-\coloneqq \min f$, and $\mu_+\coloneqq \max f$; if $G(\mu_+)=0$, we know that $\mu_+$ is the quantity of interest. Otherwise, proceed iteratively in the following way. Compute the half-way point $\mu\coloneqq \nicefrac{1}{2}(\mu_+-\mu_-)$; then, if $G(\mu)\geq 0$ set $\mu_-\coloneqq \mu$, otherwise set $\mu_+\coloneqq\mu$; then repeat. Clearly, the interval $[\mu_-,\mu_+]$ still contains the maximum root after each step. The procedure can be terminated whenever $(\mu_+-\mu_-)<\epsilon$, for some desired numerical precision $\epsilon>0$. Since the width of the interval is halved at each iteration, the runtime of this procedure is $O\bigl(\log\{(\max f - \min f)\epsilon^{-1}\}\bigr)$.
Methods for improving the numerical stability of this procedure can be found in \cite[Section 2.7.3]{de2015credal}. 

\subsection{Functions on a Single Time Point}\label{sec:funcs_single_time}

Having discussed an efficient method to find the maximum root of the function $G(\mu)$ in Section~\ref{sec:gbr}, it now remains to provide an efficient method to numerically \emph{evaluate} this function for a given value of $\mu$. Clearly, any such method will depend on the choice of $f$.

We focus on a particularly useful special case, which can be used to compute the updated lower expectation of a function $f\in\gambles(\states_s)$ on a single time point $s$, given observations at time points $u$. If $s\notin u$, then it will be notationally convenient to define $g_s\coloneqq f - \mu$, and to let $u'\coloneqq u\cup \{s\}$. We can then simply focus on computing
\begin{equation*}
\lexp\left[ \left(\prod_{t\in u}g_{t}(X_{t})\right)\bigl(f(X_s) - \mu\bigr) \right] = \lexp\left[ \prod_{t\in u'}g_{t}(X_{t})\right]\,.
\end{equation*}
On the other hand, if $s = t$ for some $t\in u$, we let $u'\coloneqq u$ and replace $g_{t}$ by $(f-\mu)g_{t}$. Clearly, the above equality then also holds; the point is simply to establish a uniform indexing notation over all time-points and functions. The right hand side of the above equality can now be computed using the following dynamic programming technique. 

For all $t\in u'$, we define auxiliary functions $g_{t}^+,g_{t}^-\in\gambles(\states_{t})$, as follows. Writing $u'=t_0,\ldots,t_{n}$, let $g_{t_{n}}^+\coloneqq g_{t_{n}}^-\coloneqq g_{t_{n}}$. Next,  for all $i\in\{0,\dots,n-1\}$ and all $x_{t_i}\in\states_{t_i}$, let
\begin{equation*}
g_{t_i}^+(x_{t_i}) \coloneqq \left\{\begin{array}{ll}
g_{t_i}(x_{t_i})\underline{\mathbb{E}}_\rateset[g_{t_{i+1}}^+(X_{t_{i+1}})\,\vert\, X_{t_i}=x_{t_i}] & \text{if $g_{t_i}(x_{t_i})\geq 0$,} \\
g_{t_i}(x_{t_i})\overline{\mathbb{E}}_\rateset[g_{t_{i+1}}^-(X_{t_{i+1}})\,\vert\, X_{t_i}=x_{t_i}] & \text{if $g_{t_i}(x_{t_i})<0$}
\end{array}\right.\vspace{-5pt}
\end{equation*}
and
\begin{equation*}
g_{t_i}^-(x_{t_i}) \coloneqq \left\{\begin{array}{ll}
g_{t_i}(x_{t_i})\overline{\mathbb{E}}_\rateset[g_{t_{i+1}}^-(X_{t_{i+1}})\,\vert\, X_{t_i}=x_{t_i}] & \text{if $g_{t_i}(x_{t_i})\geq 0$,} \\
g_{t_i}(x_{t_i})\underline{\mathbb{E}}_\rateset[g_{t_{i+1}}^+(X_{t_{i+1}})\,\vert\, X_{t_i}=x_{t_i}] & \text{if $g_{t_i}(x_{t_i})<0$.}
\end{array}\right.\vspace{5pt}
\end{equation*}
Clearly, backward recursion allows us to compute all these functions in a time-complexity order that is linear in the number of time points in $u'$. Practically, at each step, computing the quantities $\underline{\mathbb{E}}_\rateset[g_{t_{i+1}}^+(X_{t_{i+1}})\,\vert\, X_{t_i}=x_{t_i}]$ and $\overline{\mathbb{E}}_\rateset[g_{t_{i+1}}^-(X_{t_{i+1}})\,\vert\, X_{t_i}=x_{t_i}]$ can be done using Equation~\eqref{eq:lower_exp_in_steps} and the method described in Section~\ref{subsec:ICTMC_computations}. Due to the results in~\citep{krak2016ictmc}, each of these quantities is computable in polynomial time. So, the total complexity of computing all these functions is clearly also polynomial. We now have the following result.
\begin{proposition}\label{prop:computing_product_funcs}
For all $t\in u'$, let $g_{t}$, $g_{t}^+$ and $g_{t}^-$ be as defined above. Then the function of interest is given by
$\lexp\left[\prod_{t\in u'}g_{t}(X_{t})\right] = \lexp\left[g_{t_0}^+(X_{t_0})\right]$. Also,
$\uexp\left[\prod_{t\in u'}g_{t}(X_{t})\right]=\uexp\left[g_{t_0}^-(X_{t_0})\right]$.
\end{proposition}
So, in order to evaluate the function of interest, it remains to compute $\lexp\left[g_{t_0}^+(X_{t_0})\right]$. Since $g_{t_0}^+$ is a function on a single time point $t_0$, this can again be done in polynomial time, using Equation~\eqref{eq:unconditional_lower_exp}.

\section{Conclusions and Future Work}\label{sec:conclusions}

We considered the problem of performing inference with \emph{imprecise continuous-time hidden Markov chains}; an extension of \emph{imprecise continuous-time Markov chains} obtained by augmenting them with random \emph{output} variables, which may be either discrete or continuous.
Our main result is an efficient, polynomial runtime, algorithm to compute lower expectations of functions that depend on the state-space at any given time-point, given a collection of observations of the output variables. %This algorithm can be used both when the outputs are discrete and when they are continuous.

In future work, we intend to further generalise this model, by also allowing for imprecise output variables. Furthermore, we also aim to develop algorithms for other inference problems,
such as the problem of computing updated lower expectations of functions $f\in\gambles(\states_v)$ that depend on more than one time-point. Similarly, we aim to investigate predictive output inferences, i.e., the lower probability/density of observations, which has uses in classification problems. Another such problem is that of estimating state-sequences given observed output-sequences---as was previously done for (discrete-time) iHMM's~\citep{DeBock:2014ts}. %We believe that these previous results should translate fairly naturally to the current setting.
\vspace{-9pt}

\appendix

\acks{The work in this paper was partially supported by the Research Foundation - Flanders (FWO) and the H2020-MSCA-ITN-2016 UTOPIAE, grant agreement 722734. The authors would also like to thank three anonymous referees for their helpful comments and suggestions.}
\vspace{-4pt}
%\bibliography{general}

\newpage

\section{Extended Preliminaries}

We here provide some additional notation and previous results that we use throughout the proofs of our results.

\subsection{Additional Notation}

For any sequence of time points $u=\{t_0,\ldots,t_n\}$, we write for any $t\in\realsnonneg$ that $u<t$ whenever $t_i<t$ for all $i\in\{0,\ldots,n\}$. Similarly, for any $u,v\in\mathcal{U}$, we write $u< v$ when all time-points in $u$ are strictly less than all time-points in $v$.

Recall that for any $u\in\mathcal{U}$, we denote with $\gambles(\states_u)$ the set of all real-valued functions on $\states_u$. We endow these function spaces with the $L^\infty$-norm, i.e. the norm $\norm{f}$ of any $f\in\gambles(\states_u)$ is defined to be $\norm{f}\coloneqq\norm{f}_\infty=\max\{\abs{f(x_u)}\,:\,x_u\in\states_u\}$. Limits of functions are to be interpreted under this norm.

Furthermore, we sometimes use the shorthand notation $\{a_i\}_{i\in\nats}\to c$ for convergent sequences of quantities, which should be read as $\lim_{i\to+\infty}a_i=c$. If this limit is approached from above or below, we write $\{a_i\}_{i\in\nats}\to c^+$ or $\{a_i\}_{i\in\nats}\to c^-$, respectively.

\subsection{A Useful Property of ICTMC's}

The below states a very useful property of the lower expectation corresponding to ICTMC's, that we will require in some of our proofs.

\begin{lemma}[Iterated Lower Expectation]\cite[Theorem 6.5]{krak2016ictmc}\label{lemma:iterated_lower}
Let $\rateset$ be a non-empty, bounded and convex set of rate matrices, and let $\mathcal{M}$ be a non-empty set of probability mass functions on $\states_0$. Let $\mathbb{P}_{\rateset,\mathcal{M}}$ denote the corresponding ICTMC.
Let $u\subset\realsnonneg$ be a finite (possibly empty) sequence of time-points, and consider any $v,w\in\mathcal{U}$ such that $u<v<w$. Choose any $f\in\gambles(\states_{u\cup v\cup w})$. Then,
\begin{equation*}
\lexp[f(X_u,X_v,X_w)\,\vert\,X_u] = \lexp\bigl[\lexp[f(X_u,X_v,X_w)\,\vert\,X_u,X_v]\,\vert\,X_u\bigr]\,.
\end{equation*}
\end{lemma}

\section{Globally Required Proofs and Lemmas}

The following property will be useful. The result is rather trivial, but we note it here explicitly to prevent confusion when we use it in our proofs.
\begin{lemma}\label{lemma:limit_exp_is_exp_limit}
Let $P_\states$ be a stochastic process and consider any $u\in\mathcal{U}$ and any $\{f_i\}_{i\in\nats}\to f$ in $\gambles(\states_u)$. Then $\lim_{i\to+\infty}\mathbb{E}_{P_\states}[f_i(X_u)]=\mathbb{E}_{P_\states}[f(X_u)]$.
\end{lemma}
\begin{proof}
Trivial consequence of the definition of our norm $\norm{\cdot}$ on $\gambles(\states_u)$.
\end{proof}

The following lemma states the imprecise analogue of the above result; this is essentially well-known, but we repeat it here for the sake of completeness.

\begin{lemma}\label{lemma:limit_lexp_is_lexp_limit}
Let $\mathbb{P}_{\rateset,\mathcal{M}}$ be an ICTMC and consider any $u\in\mathcal{U}$ and any $\{f_i\}_{i\in\nats}\to f$ in $\gambles(\states_u)$. Then $\lim_{i\to+\infty}\lexp[f_i(X_u)]=\lexp[f(X_u)]$.
\end{lemma}
\begin{proof}
Keeping $u$ fixed, then since $\lexp[\,\cdot\,]\,:\,\gambles(\states_u)\to\reals$ is an infimum over precise expectations $\mathbb{E}_{P_\states}[\,\cdot\,]\,:\,\gambles(\states_u)\to\reals$, with $P_\states\in\mathbb{P}_{\rateset,\mathcal{M}}$, we know from~\cite[Theorem 3.3.3]{Walley:1991vk} that $\lexp[\,\cdot\,]$ is a coherent lower prevision on $\gambles(\states_u)$. Therefore, the statement follows directly from~\cite[Proposition 2.6.1.$\ell$]{Walley:1991vk} and the definition of our norm $\norm{\cdot}$ on $\gambles(\states_u)$.
\end{proof}

We provide the proof of Proposition~\ref{prop:GBR_properties} below; this is not in chronological order with respect to the main text, but it states a number of convenient properties that are required in the proofs of statements that appear before Proposition~\ref{prop:GBR_properties}. We first need the following lemma.

\begin{lemma}\label{lemma:conditioning_zero_means_bayes_zero}
Let $P_\states$ be a stochastic process and consider any $u,v\in\mathcal{U}$, any $f\in\gambles(\states_v)$ and any $g\in\gambles(\states_u)$ such that $g\geq 0$. If\/ $\mathbb{E}_{P_\states}[g(X_u)]=0$, then for all $\mu\in\reals$ it holds that
\begin{equation*}
\mathbb{E}_{P_\states}\bigl[g(X_u)\bigl(f(X_v) - \mu\bigr)\bigr]=0\,.
\end{equation*}
\end{lemma}
\begin{proof}
Because $\mathbb{E}_{P_\states}[g(X_u)]=0$, and since $g\geq 0$, we must clearly have that
\begin{equation*}
P_\states(X_u=x_u) = 0\,,
\end{equation*}
for all $x_u\in\states_u$ for which $g(x_u)\neq 0$. Let $\states_u^0\coloneqq\{x_u\in\states_u\,:\,g(x_u)\neq 0\}$. Then clearly for any $x_u\in\states_u^0$ it holds for all $x_{v\setminus u}\in\states_{v\setminus u}$ that also $P_\states(X_u=x_u,X_{v\setminus u}=x_{v\setminus u})=0$. Hence, for all $x_u\in\states_u$ and $x_{v\setminus u}\in\states_{v\setminus u}$, we find that $P_\states(X_u=x_u,X_{v\setminus u}=x_{v\setminus u})g(x_u)=0$. Therefore,
\begin{align*}
\mathbb{E}_{P_\states}\bigl[g(X_u)\bigl(f(X_v) - \mu\bigr)\bigr] &= \sum_{x_{u\cup v}\in\states_{u\cup v}} P_\states(X_{u\cup v}=x_{u\cup v})g(x_u)\bigl(f(x_v) - \mu\bigr) \\
 &= \sum_{x_{u}\in\states_{u}}\sum_{x_{v\setminus u}\in\states_{v\setminus u}} P_\states(X_u=x_u,X_{v\setminus u}=x_{v\setminus u})g(x_u)\bigl(f(x_v) - \mu\bigr)=0. 
\end{align*}
\end{proof}

\begin{proof}{\bf of Proposition~\ref{prop:GBR_properties}~}
For brevity, define $g\in\gambles(\states_u)$ as $g(x_u)\coloneqq \prod_{t\in u} g_{t}(x_{t})$ for all $x_u\in\states_u$.

We start by proving Property~\ref{GBR:always}. For continuity, consider any $\mu\in\reals$. We will prove that $G$ is continuous in $\mu$, or in other words, that for every sequence $\{\mu_i\}_{i\in\nats}\to\mu$ it holds that $\lim_{i\to\infty}G(\mu_i)=G(\mu)$. So, choose any sequence $\{\mu_i\}_{i\in\nats}\to\mu$, and consider the induced sequence of functions $\bigl\{g(X_u)\bigl(f(X_v) - \mu_i\bigr)\bigr\}_{i\in\nats}$ in $\gambles(\states_{u\cup v})$. Then, since $\{\mu_i\}_{i\in\nats}\to\mu$, clearly also $\lim_{i\to+\infty}g(X_u)\bigl(f(X_v) - \mu_i\bigr)=g(X_u)\bigl(f(X_v) - \mu\bigr)$. Using Lemma~\ref{lemma:limit_lexp_is_lexp_limit}, we therefore find that
\begin{equation*}
\lim_{i\to+\infty}\lexp[g(X_u)\bigl(f(X_v) - \mu_i\bigr)] = \lexp[g(X_u)\bigl(f(X_v) - \mu\bigr)]\,,
\end{equation*}
or in other words, that $\lim_{i\to+\infty}G(\mu_i) = G(\mu)$. Since the sequence $\{\mu_i\}_{i\in\nats}$ was arbitrary, this concludes the proof.

We next prove that $G$ is non-increasing. To this end, fix any $P_\states\in\mathbb{P}_{\rateset,\mathcal{M}}$. Then, by the linearity of expectation operators,
\begin{align}\label{eq:gbr_linear_expansion}
\mathbb{E}_{P_\states}\bigl[g(X_u)\bigl(f(X_v) - \mu\bigr)\bigr] = \mathbb{E}_{P_\states}[g(X_u)f(X_v)] - \mu\mathbb{E}_{P_\states}[g(X_u)]\,.
\end{align}
Since by assumption $g\geq 0$, it must hold that $\mathbb{E}_{P_\states}[g(X_u)]\geq 0$, and so, that $\mathbb{E}_{P_\states}\bigl[g(X_u)\bigl(f(X_v) - \mu\bigr)\bigr]$ is non-increasing in $\mu$. Since this is true for all ${P_\states}\in\mathbb{P}_{\rateset,\mathcal{M}}$, we find that $G(\mu)$ is a lower envelope of non-increasing functions, which must therefore be non-increasing itself. This concludes the proof.

For concavity, fix any $\mu,\nu\in\reals$, and choose any $\lambda\in[0,1]$. Let $\mu'\coloneqq \lambda\mu + (1-\lambda)\nu$. We need to show that $\lambda G(\mu) + (1-\lambda)G(\nu) \leq G(\mu')$. To this end, fix any $\epsilon\in\realspos$. Then, there is some ${P_\states}\in\mathbb{P}_{\rateset,\mathcal{M}}$ such that
\begin{equation}\label{eq:gbr_prop:convex_approached}
\mathbb{E}_{P_\states}\left[g(X_u)\bigl(f(X_v) - \mu'\bigr)\right] - \epsilon < G(\mu').
\end{equation}
By expanding the convex combination $\mu'$ using the linearity of expectation operators, we find that
\begin{align*}
&\mathbb{E}_{P_\states}\left[g(X_u)\bigl(f(X_v) - \mu'\bigr)\right]\\  &= \mathbb{E}_{P_\states}\left[g(X_u)f(X_v)\right] - \mu'\mathbb{E}_{P_\states}\left[g(X_u)\right] \\
 &= \mathbb{E}_{P_\states}\left[g(X_u)f(X_v)\right] - \lambda\mu\mathbb{E}_{P_\states}\left[g(X_u)\right] - (1-\lambda)\nu\mathbb{E}_{P_\states}\left[g(X_u)\right] \\
 &= \bigl(\lambda + (1-\lambda)\bigr)\mathbb{E}_{P_\states}\left[g(X_u)f(X_v)\right] - \lambda\mu\mathbb{E}_{P_\states}\left[g(X_u)\right] - (1-\lambda)\nu\mathbb{E}_{P_\states}\left[g(X_u)\right] \\
 &= \lambda\mathbb{E}_{P_\states}\left[g(X_u)f(X_v)\right] - \lambda\mu\mathbb{E}_{P_\states}\left[g(X_u)\right]
+(1-\lambda)\mathbb{E}_{P_\states}\left[g(X_u)f(X_v)\right] - (1-\lambda)\nu\mathbb{E}_{P_\states}\left[g(X_u)\right] \\
 &= \lambda\mathbb{E}_{P_\states}\left[g(X_u)\bigl(f(X_v) - \mu\bigr)\right] + (1-\lambda)\mathbb{E}_{P_\states}\left[g(X_u)\bigl(f(X_v) - \nu\bigr)\right] \\
 &\geq \lambda\lexp\left[g(X_u)\bigl(f(X_v) - \mu\bigr)\right] + (1-\lambda)\lexp\left[g(X_u)\bigl(f(X_v) - \nu\bigr)\right] \\
 &= \lambda G(\mu) + (1-\lambda) G(\nu)\,,
\end{align*}
where the inequality follows from the fact that $\lambda$ and $(1-\lambda)$ are non-negative, and $P_\states\in\mathbb{P}_{\rateset,\mathcal{M}}$---hence in particular $\lexp\leq\mathbb{E}_{P_\states}$.
Combining with Equation~\eqref{eq:gbr_prop:convex_approached}, we find that
\begin{equation*}
\lambda G(\mu) + (1-\lambda) G(\nu) -\epsilon \leq \mathbb{E}_{P_\states}\left[g(X_u)\bigl(f(X_v) - \mu'\bigr)\right] - \epsilon < G(\mu')\,.
\end{equation*}
Since the $\epsilon\in\realspos$ was arbitrary, this concludes the proof.

To prove that the function has a root, first consider any $\mu<\min f$. Then, for every $P_\states\in\mathbb{P}_{\rateset,\mathcal{M}}$, using the assumption that $g\geq 0$,
\begin{equation*}
\mathbb{E}_{P_\states}[g(X_u)f(X_v)] \geq \mathbb{E}_{P_\states}[g(X_u)\min f] = \mu\mathbb{E}_{P_\states}[g(X_u)]\,,
\end{equation*}
and hence, using Equation~\eqref{eq:gbr_linear_expansion}, we find that $\mathbb{E}_{P_\states}\bigl[g(X_u)\bigl(f(X_v)-\mu\bigr)\bigr] \geq 0$. Since this is true for all $P_\states\in \mathbb{P}_{\rateset,\mathcal{M}}$, we find that $G(\mu)\geq 0$ for this choice of $\mu$. Next, consider any $\nu>\max f$. Then by a completely analogous argument---just reverse the inequalities---we find that $G(\nu)\leq 0$ for this choice of $\nu$. Therefore, and since we already know that $G$ is continuous, by the intermediate value theorem there must now be some $\mu_*\in[\mu,\nu]$ such that $G(\mu_*)=0$. This concludes the proof.

We next prove \ref{GBR:low_pos}, and start by showing that $G$ is (strictly) decreasing if $\lexp[g(X_u)]>0$. To this end, consider any $\mu\in\reals$ and any $\Delta\in\realspos$. We need to show that $G(\mu)>G(\mu+\Delta)$. Let $\epsilon\coloneqq \Delta \lexp[g(X_u)]$; clearly then $\epsilon>0$. Therefore, there is some $P_\states\in\mathbb{P}_{\rateset,\mathcal{M}}$ such that
\begin{equation}\label{eq:gbr_prop:approached_closeby}
\mathbb{E}_{P_\states}\bigl[g(X_u)\bigl(f(X_v)-\mu\bigr)\bigr] - \epsilon < G(\mu)\,.
\end{equation}
Expanding the left hand side as in Equation~\eqref{eq:gbr_linear_expansion}, we find
\begin{align*}
\mathbb{E}_{P_\states}\bigl[g(X_u)\bigl(f(X_v)-\mu\bigr)\bigr] - \epsilon &= \mathbb{E}_{P_\states}\bigl[g(X_u)f(X_v)\bigr] - \mu \mathbb{E}_{P_\states}\bigl[g(X_u)\bigr] - \epsilon \\
 &= \mathbb{E}_{P_\states}\bigl[g(X_u)f(X_v)\bigr] - \mu \mathbb{E}_{P_\states}\bigl[g(X_u)\bigr] - \Delta\lexp[g(X_u)] \\
 &\geq \mathbb{E}_{P_\states}\bigl[g(X_u)f(X_v)\bigr] - \mu \mathbb{E}_{P_\states}\bigl[g(X_u)\bigr] - \Delta\mathbb{E}_{P_\states}\bigl[g(X_u)\bigr] \\
 &= \mathbb{E}_{P_\states}\bigl[g(X_u)f(X_v)\bigr] - (\mu+\Delta) \mathbb{E}_{P_\states}\bigl[g(X_u)\bigr] \\
 &= \mathbb{E}_{P_\states}\bigl[g(X_u)\bigl(f(X_v)-(\mu+\Delta)\bigr)\bigr] \\
 &\geq G(\mu+\Delta)\,,
\end{align*}
where the first inequality follows from the fact that $\Delta$ and $\lexp[g(X_u)]$ are strictly positive, and $\lexp[g(X_u)]\leq \mathbb{E}_{P_\states}\bigl[g(X_u)\bigr]$ since $P_\states\in\mathbb{P}_{\rateset,\mathcal{M}}$. Combining with Equation~\eqref{eq:gbr_prop:approached_closeby} shows that
\begin{equation*}
G(\mu+\Delta) \leq \mathbb{E}_{P_\states}\bigl[g(X_u)\bigl(f(X_v)-\mu\bigr)\bigr] - \epsilon < G(\mu)\,,
\end{equation*}
which concludes the proof.

To prove that $G$ has a \emph{unique} root under the assumption that $\lexp[g(X_u)]>0$, note that we already know that $G$ has at least one root, i.e. $G(\mu)=0$ for some $\mu\in\reals$. By combining with the fact that $G$ is strictly decreasing under the assumption that $\lexp[g(X_u)]>0$, the uniqueness of this root follows immediately.

We next prove~\ref{GBR:up_pos}. Lemma~\ref{lemma:general_regular_extension} below states that $G$ has a maximum root if $\uexp[g(X_u)]>0$, so we do not need to prove this here. So, let $\mu_*\coloneqq \max\{\mu\in\reals\,:\,G(\mu)\geq 0\}$ be this maximum root.

We will now prove that $G(\mu)=0$ for all $\mu\leq\mu_*$ when $\lexp[g(X_u)]=0$ but $\uexp[g(X_u)]>0$. So, consider any $\mu<\mu_*$ (the case for $\mu=\mu_*$ is trivial). Since $\mu<\mu_*$, we already know that $G(\mu)$ is non-negative because $G$ is non-increasing; so, it suffices to show that $G(\mu)$ is non-positive. Now, for every $P_\states\in\mathbb{P}_{\rateset,\mathcal{M}}$ we have that
\begin{align}
\mathbb{E}_{P_\states}\bigl[g(X_u)\bigl(f(X_v) - \mu\bigr)\bigr] 
&= \sum_{x_{u\cup v}} P_\states(X_{u\cup v}=x_{u\cup v})g(x_u)\bigl(f(x_v)-\mu\bigr) \notag\\
 &\leq \sum_{x_{u\cup v}} P_\states(X_{u\cup v}=x_{u\cup v})g(x_u)\abs{f(x_v)-\mu} \notag\\
 &\leq \sum_{x_{u\cup v}} P_\states(X_{u\cup v}=x_{u\cup v})g(x_u)\norm{f-\mu} \notag\\
 &= \sum_{x_{u}} P_\states(X_{u}=x_{u})g(x_u)\norm{f-\mu}= \norm{f-\mu}\mathbb{E}_{P_\states}\bigl[g(X_u)\bigr],
\label{eq:gbr_prop:bound_value}
\end{align}
where the first inequality follows from the fact that $g\geq 0$, the second inequality is due to the definition of the norm $\norm{\cdot}$, and the second equality follows from the law of total probability (that is, $X_{v\setminus u}$ is marginalised out). Since $G(\mu)\leq \mathbb{E}_{P_\states}\bigl[g(X_u)\bigl(f(X_v) - \mu\bigr)\bigr]$, it follows from the above that if $\norm{f-\mu}=0$ then $G(\mu)\leq 0$, in which case we are done. 

Hence, we can assume without loss of generality that $\norm{f-\mu}>0$. Choose any $\epsilon\in\realspos$. Then, because $\lexp[g(X_u)]=0$, there is some $P_\states\in\mathbb{P}_{\rateset,\mathcal{M}}$ such that $\mathbb{E}_{P_\states}[g(X_u)] < \nicefrac{\epsilon}{\norm{f-\mu}}$, which implies that
\begin{equation*}
\epsilon > \norm{f-\mu}\mathbb{E}_{P_\states}[g(X_u)] \geq \mathbb{E}_{P_\states}\bigl[g(X_u)\bigl(f(X_v) - \mu\bigr)\bigr] \geq G(\mu)\,,
\end{equation*}
using Equation~\eqref{eq:gbr_prop:bound_value} for the second inequality. Since the $\epsilon\in\realspos$ was arbitrary, this implies that $G(\mu)$ is non-positive, which concludes the proof.

We next show that $G$ is strictly decreasing for $\mu>\mu_*$. So consider any $\mu\in\reals$ such that $\mu>\mu_*$, and any $\Delta\in\realspos$; we need to show that $G(\mu)>G(\mu+\Delta)$. Because $\mu>\mu_*$, and since $\mu_*=\max\{\mu\in\reals\,:\,G(\mu)\geq 0\}$, we know that $G(\mu)<0$.

First note that for any $\epsilon\in\realspos$, there is some $P_{\states}\in\mathbb{P}_{\rateset,\mathcal{M}}$ for which
\begin{equation*}
G(\mu) > \mathbb{E}_{P_\states}\bigl[g(X_u)\bigl(f(X_v)-\mu\bigr)\bigr] - \epsilon\,,
\end{equation*}
and clearly also
\begin{equation*}
\mathbb{E}_{P_\states}\bigl[g(X_u)\bigl(f(X_v)-\mu_*\bigr)\bigr] \geq G(\mu_*) = 0\,.
\end{equation*}
Therefore,
\begin{align*}
G(\mu)+\epsilon > \mathbb{E}_{P_\states}\bigl[g(X_u)\bigl(f(X_v)-\mu\bigr)\bigr] 
 &\geq \mathbb{E}_{P_\states}\bigl[g(X_u)\bigl(f(X_v)-\mu\bigr)\bigr] - \mathbb{E}_{P_\states}\bigl[g(X_u)\bigl(f(X_v)-\mu_*\bigr)\bigr] \\
 &= -\mu\mathbb{E}_{P_\states}[g(X_u)] + \mu_*\mathbb{E}_{P_\states}[g(X_u)]
 = (\mu_*-\mu)\mathbb{E}_{P_\states}[g(X_u)]
\end{align*}
and so, negating both sides and noting again that $G(\mu)<0$,
\begin{equation*}
(\mu-\mu_*)\mathbb{E}_{P_\states}[g(X_u)] > \abs{G(\mu)}-\epsilon\,,
\end{equation*}
and dividing through,
\begin{equation}\label{eq:gbr_prop:lower_bound_slope}
\mathbb{E}_{P_\states}[g(X_u)] > \frac{\abs{G(\mu)}-\epsilon}{\mu-\mu_*}\,,
\end{equation}
for any $P_\states$ that $\epsilon$-approaches $G(\mu)$. 

The idea is now the following: Equation~\eqref{eq:gbr_prop:lower_bound_slope} provides a lower bound on the (absolute magnitude of the) slope of the function $\mathbb{E}_{P_\states}\bigl[g(X_u)\bigl(f(X_v)-\mu\bigr)\bigr]$ for any $P_\states$ that $\epsilon$-approaches $G(\mu)$. Since this function can still be a distance of $\epsilon$ above $G(\mu)$, we now need to make sure that the slope is such that this function will decrease by more than $\epsilon$ after increasing $\mu$ by $\Delta$; this will guarantee that this function becomes strictly lower than $G(\mu)$ when evaluated at $\mu+\Delta$, and since it is an upper bound on $G(\mu+\Delta)$, this will complete the proof.

Note that this lower bound increases as we decrease $\epsilon$. In order to ensure that we get a large enough slope, we now solve the following for $\epsilon$:
\begin{align}
\Delta\frac{\abs{G(\mu)}-\epsilon}{\mu-\mu_*} > \epsilon 
&\asa
\frac{\Delta\abs{G(\mu)}}{\mu-\mu_*} - \frac{\Delta\epsilon}{\mu-\mu_*} > \epsilon\notag\\
&\asa
\frac{\Delta\abs{G(\mu)}}{\mu-\mu_*} > \epsilon + \frac{\Delta\epsilon}{\mu-\mu_*} \notag\\
&\asa\Delta\abs{G(\mu)} > (\mu-\mu_*)\epsilon + \Delta\epsilon \notag\\
&\asa\Delta\abs{G(\mu)} > \epsilon\bigl((\mu-\mu_*) + \Delta\bigr)
\asa
\frac{\Delta\abs{G(\mu)}}{\Delta + (\mu-\mu_*)} > \epsilon.\label{eq:gbr_prop:solving_for_slope}
\end{align}
So, choose any $\epsilon\in\realspos$ such that $\epsilon< \nicefrac{\Delta\abs{G(\mu)}}{(\Delta + \mu-\mu_*)}$; since the right-hand side is clearly strictly positive, this is always possible. Then there is some $P_\states\in\mathbb{P}_{\rateset,\mathcal{M}}$ such that
\begin{equation*}
\mathbb{E}_{P_\states}\bigl[g(X_u)\bigl(f(X_v)-\mu\bigr)\bigr] - \epsilon < G(\mu)\,.
\end{equation*}
For this same $P_\states$, we then have
\begin{align*}
G(\mu+\Delta) &\leq \mathbb{E}_{P_\states}\bigl[g(X_u)\bigl(f(X_v)-(\mu+\Delta)\bigr)\bigr] \\
 &= \mathbb{E}_{P_\states}\bigl[g(X_u)\bigl(f(X_v)-\mu\bigr)\bigr] - \Delta\mathbb{E}_{P_\states}\bigl[g(X_u)\bigr] \\
 &< \mathbb{E}_{P_\states}\bigl[g(X_u)\bigl(f(X_v)-\mu\bigr)\bigr] - \Delta\frac{\abs{G(\mu)}-\epsilon}{\mu-\mu_*} \\
 &< \mathbb{E}_{P_\states}\bigl[g(X_u)\bigl(f(X_v)-\mu\bigr)\bigr] - \epsilon \\
 &< G(\mu)\,,
\end{align*}
where the second inequality is by Equation~\eqref{eq:gbr_prop:lower_bound_slope}, the third inequality is by Equation~\eqref{eq:gbr_prop:solving_for_slope}, and the final inequality is by the choice of $P_\states$. So, we have found that indeed $G(\mu)>G(\mu+\Delta)$, which concludes the proof.

We finally prove~\ref{GBR:none_pos}, i.e. that $G(\mu)$ is identically zero whenever $\uexp[g(X_u)]=0$. Clearly, this assumption implies that $\mathbb{E}_{P_\states}[g(X_u)]=0$ for all $P_\states\in\mathbb{P}_{\rateset,\mathcal{M}}$. Therefore, and since $g\geq 0$, it follows from Lemma~\ref{lemma:conditioning_zero_means_bayes_zero} that for any $\mu\in\reals$, it holds that $\mathbb{E}_{P_\states}\bigl[g(X_u)\bigl(f(X_v) - \mu\bigr)\bigr]=0$ for all $P_\states\in\mathbb{P}_{\rateset,\mathcal{M}}$. It immediately follows that indeed $G(\mu)=0$ for all $\mu\in\reals$.
\end{proof}

The following lemma states a more general version of the result that the generalised Bayes' rule computes the updated lower expectation of a model under regular extension. We state it here because we use the result for various proofs throughout this appendix.

\begin{lemma}\label{lemma:general_regular_extension}
Let $\mathbb{P}_{\rateset,\mathcal{M}}$ be an ICTMC, and consider any $u,v\in\mathcal{U}$, $f\in\gambles(\states_v)$ and $g\in\gambles(\states_u)$ such that $g\geq 0$. Then, if $\uexp[g(X_u)]>0$, it holds that
\begin{multline*}
 \max\left\{\mu\in\reals\,:\, \lexp\bigl[g(X_u)\bigl(f(X_v) - \mu\bigr)\bigr] \geq 0\right\} \\
 = \inf\left\{ \frac{\mathbb{E}_{P_\states}[f(X_v)g(X_u)]}{\mathbb{E}_{P_\states}[g(X_u)]}\,:\,{P_\states}\in\mathbb{P}_{\rateset,\mathcal{M}}, \mathbb{E}_{P_\states}[g(X_u)]>0 \right\}\,.
\end{multline*}
\end{lemma}
\begin{proof}
Let $\mathcal{P}\coloneqq \left\{ {P_\states}\in\mathbb{P}_{\rateset,\mathcal{M}}\,:\, \mathbb{E}_{P_\states}[g(X_u)] > 0\right\}$, and note that $\mathcal{P}$ is non-empty due to the assumption that $\uexp[g(X_u)]>0$. For all ${P_\states}\in\mathcal{P}$, define
\begin{equation*}
\mu_{P_\states} \coloneqq \frac{\mathbb{E}_{P_\states}[f(X_v)g(X_u)]}{\mathbb{E}_{P_\states}[g(X_u)]}\,,
\end{equation*}
and let $\mu_*$ be defined by
\begin{equation*}
\mu_* \coloneqq \inf\left\{\mu_{P_\states}\,:\,{P_\states}\in\mathcal{P} \right\}\,.
\end{equation*}
Now define the following function, parameterised in $\mu\in\reals$,
\begin{equation*}
\underline{\mathbb{E}}[g(X_u)(f(X_v) - \mu)] \coloneqq \inf\{\mathbb{E}_{P_\states}[g(X_u)(f(X_v) - \mu)]\,:\,{P_\states}\in\mathcal{P} \}\,,
\end{equation*}
and consider $\underline{\mathbb{E}}[g(X_u)(f(X_v) - \mu_*)]$. We start by showing that this quantity is non-negative. To this end, fix any $\epsilon>0$. Then, there is some ${P_\states}\in\mathcal{P}$ such that
\begin{equation*}
\mathbb{E}_{P_\states}[g(X_u)(f(X_v) - \mu_*)] - \epsilon < \underline{\mathbb{E}}[g(X_u)(f(X_v) - \mu_*)]\,.
\end{equation*}
Using Equation~\eqref{eq:gbr_linear_expansion}, the function $\mathbb{E}_{P_\states}[g(X_u)(f(X_v) - \mu)]$ is non-increasing in $\mu$. Therefore, and since $\mu_*\leq \mu_{P_\states}$, we have
\begin{equation*}
\mathbb{E}_{P_\states}[g(X_u)(f(X_v) - \mu_{P_\states})] - \epsilon \leq \mathbb{E}_{P_\states}[g(X_u)(f(X_v) - \mu_*)] - \epsilon < \underline{\mathbb{E}}[g(X_u)(f(X_v) - \mu_*)]\,.
\end{equation*}
Due to the choice of $\mu_{P_\states}$, we have $\mathbb{E}_{P_\states}[g(X_u)(f(X_v) - \mu_{P_\states})]=0$, and so we find that
\begin{equation*}
-\epsilon < \underline{\mathbb{E}}[g(X_u)(f(X_v) - \mu_*)]\,.
\end{equation*}
Since this is true for every $\epsilon>0$, we conclude that $0\leq \underline{\mathbb{E}}[g(X_u)(f(X_v) - \mu_*)]$. Next, we show that this quantity is also non-positive, or in other words, that $\mu_*$ is a root of this function. 

To this end, fix any $\epsilon>0$, and define $\epsilon'\coloneqq \nicefrac{\epsilon}{\uexp[g(X_u)]}$; since by assumption $\uexp[g(X_u)]>0$, we have $\epsilon' >0$. Now consider ${P_\states}\in\mathcal{P}$ such that
\begin{equation*}
\mu_{P_\states} - \epsilon' < \mu_*\,.
\end{equation*}
Then, since $\underline{\mathbb{E}}[g(X_u)(f(X_v) - \mu)]$ is non-increasing in $\mu$---because it is a lower envelope of non-increasing functions---we have that
\begin{equation*}
\underline{\mathbb{E}}[g(X_u)(f(X_v) - \mu_*)] \leq \underline{\mathbb{E}}[g(X_u)(f(X_v) - (\mu_{P_\states} - \epsilon'))]\leq \mathbb{E}_{P_\states}[g(X_u)(f(X_v) - (\mu_{P_\states} - \epsilon'))]\,.
\end{equation*}
Expanding the r.h.s. using the linearity of expectation operators, and by the definition of $\mu_{P_\states}$, we then have
\begin{multline*}
\mathbb{E}_{P_\states}[g(X_u)(f(X_v) - (\mu_{P_\states} - \epsilon'))]\\ = \mathbb{E}_{P_\states}[g(X_u)f(X_v)] - \mu_{P_\states}\mathbb{E}_{P_\states}[g(X_u)] + \epsilon'\mathbb{E}_{P_\states}[g(X_u)] = \epsilon'\mathbb{E}_{P_\states}[g(X_u)]\,,
\end{multline*}
and since ${P_\states}\in\mathcal{P}\subseteq\mathbb{P}_{\rateset,\mathcal{M}}$,
\begin{equation*}
\epsilon'\mathbb{E}_{P_\states}[g(X_u)] \leq \epsilon'\uexp[g(X_u)]=\epsilon\,,
\end{equation*}
and so we find that
\begin{equation*}
\underline{\mathbb{E}}[g(X_u)(f(X_v) - \mu_*)] \leq \epsilon\,.
\end{equation*}
Since this is true for every $\epsilon>0$, and since we already know that $\underline{\mathbb{E}}[g(X_u)(f(X_v) - \mu_*)]$ is non-negative, we conclude that
\begin{equation*}
\underline{\mathbb{E}}[g(X_u)(f(X_v) - \mu_*)] = 0\,.
\end{equation*}

Now consider any $\mu' > \mu_*$. There must then be some ${P_\states}\in\mathcal{P}$ such that $\mu_*\leq \mu_{P_\states} < \mu'$, and furthermore,
\begin{equation*}
\underline{\mathbb{E}}[g(X_u)(f(X_v) - \mu')] \leq \mathbb{E}_{P_\states}[g(X_u)(f(X_v) - \mu')] < \mathbb{E}_{P_\states}[g(X_u)(f(X_v) - \mu_{P_\states})] = 0\,,
\end{equation*}
where the strict inequality follows from the fact that $\mathbb{E}_{P_\states}[g(X_u)(f(X_v) - \mu)]$ is strictly decreasing, and $\mu_{P_\states}<\mu'$. Since this is true for every $\mu'>\mu_*$, we conclude that
\begin{equation*}
\mu_* = \max\left\{ \mu\in\reals\,:\, \underline{\mathbb{E}}[g(X_u)(f(X_v) - \mu)] \geq 0 \right\}\,.
\end{equation*}
Hence, because of our definition for $\mu_*$, we are left to prove that
\begin{equation}\label{eq:addedequation}
\max\left\{ \mu\in\reals\,:\, \underline{\mathbb{E}}[g(X_u)(f(X_v) - \mu)] \geq 0 \right\}=\max\left\{ \mu\in\reals\,:\, \lexp[g(X_u)(f(X_v) - \mu)] \geq 0 \right\}.
\end{equation}
If $\mathcal{P}=\mathbb{P}_{\rateset,\mathcal{M}}$, this is trivially true.
Therefore, we can assume without loss of generality that $\mathcal{P}\neq \mathbb{P}_{\rateset,\mathcal{M}}$. Let $\mathcal{P}_0\coloneqq \mathbb{P}_{\rateset,\mathcal{M}}\setminus \mathcal{P}$. Due to the definition of $\mathcal{P}$, we then have for every $P_\states\in\mathcal{P}_0$ that $\mathbb{E}_{P_\states}[g(X_u)]=0$. It follows from Lemma~\ref{lemma:conditioning_zero_means_bayes_zero} that $\mathbb{E}_{P_\states}\bigl[g(X_u)\bigl(f(X_v)-\mu\bigr)\bigr]=0$ for all $\mu\in\reals$ and all $P_\states\in\mathcal{P}_0$. Hence,
\begin{equation*}
\inf\left\{ \mathbb{E}_{P_\states}\bigl[g(X_u)\bigl(f(X_v)-\mu\bigr)\bigr]\,:\,P_\states\in\mathcal{P}_0 \right\} = 0
\text{ for all $\mu\in\reals$.}
\end{equation*}
Because $\mathcal{P}\cup\mathcal{P}_0=\mathbb{P}_{\rateset,\mathcal{M}}$, we therefore have
\begin{equation*}
\lexp\bigl[g(X_u)\bigl(f(X_v)-\mu\bigr)\bigr] = \min\left\{\underline{\mathbb{E}}[g(X_u)(f(X_v) - \mu)],\,0 \right\}\,,
\end{equation*}
and so we conclude that $\lexp\bigl[g(X_u)\bigl(f(X_v)-\mu\bigr)\bigr]$ is negative if and only if $\underline{\mathbb{E}}[g(X_u)(f(X_v) - \mu)]$ is negative. This clearly implies Equation~\eqref{eq:addedequation}.
\end{proof}

\begin{corollary}\label{cor:lower_hidden_is_root_chain}
Let $\mathcal{Z}$ be an ICTHMC, and consider any $u,v\in\mathcal{U}$, $f\in\gambles(\states_v)$ and $g\in\gambles(\states_u)$ such that $g\geq 0$. Then, if $\uexp[g(X_u)]>0$, it holds that
\begin{multline*}
\max\left\{\mu\in\reals\,:\, \lexp\bigl[g(X_u)\bigl(f(X_v) - \mu\bigr)\bigr] \geq 0\right\}\\ = \inf\left\{ \frac{\mathbb{E}_{P_\states}[f(X_v)g(X_u)]}{\mathbb{E}_{P_\states}[g(X_u)]}\,:\,P\in\mathcal{Z}, \mathbb{E}_{P_\states}[g(X_u)]>0 \right\}\,.
\end{multline*}
\end{corollary}
\begin{proof}
Trivial consequence of Lemma~\ref{lemma:general_regular_extension} and Definition~\ref{def:hidden_ictmc}.
\end{proof}

Due to the independence assumptions on augmented stochastic processes and the fact that we use a fixed output distribution $P_{\observs\vert\states}$, (lower) probabilities of outputs can be conveniently rewritten as follows.
\begin{lemma}\label{lemma:output_probability_is_expectation}
Let $P=P_{\observs\vert\states}\otimes P_\states$ be an augmented stochastic process and consider any $u\in\mathcal{U}$ and any $O_u\in\Sigma_u$. Then $P(Y_u\in O_u)=\mathbb{E}_{P_\states}[P_{\observs\vert\states}(O_u\vert X_u)]$.
\end{lemma}
\begin{proof}
Due to the construction of augmented stochastic processes in Section~\ref{sec:aug_stochastic_processes}, we have that
\begin{equation*}
P(Y_u\in O_u) = \sum_{x_u\in\states_u} P(Y_u\in O_u, X_u=x_u) 
 = \sum_{x_u\in\states_u} P_{\observs\vert\states}(O_u\vert x_u)P_\states(X_u=x_u) 
 = \mathbb{E}_{P_\states}[P_{\observs\vert\states}(O_u\vert X_u)].
\end{equation*}
\end{proof}
\begin{lemma}\label{lemma:lower_output_probability_is_expectation}
Let $\mathcal{Z}$ be an ICTHMC with corresponding ICTMC $\mathbb{P}_{\rateset,\mathcal{M}}$ and consider any $u\in\mathcal{U}$ and $O_u\in\Sigma_u$. Then $\overline{P}_{\mathcal{Z}}(Y_u\in O_u)=\uexp[P_{\observs\vert\states}(O_u\vert X_u)]$.
\end{lemma}
\begin{proof}
Trivial consequence of Definition~\ref{def:hidden_ictmc} and Lemma~\ref{lemma:output_probability_is_expectation}.
\end{proof}

\section{Proofs of the Results in Section~\ref{subsec:pos_prob}}

\begin{proof}{\bf of Proposition~\ref{prop:precise_conditioning_for_positive}~}
The result follows from some simple manipulations, but we need to keep track of the overlap $u\cap v$ between the time-points of interest to prevent double-counting those time-points. Let $w\coloneqq u\setminus v$; then clearly $v\cup w = v\cup (u\setminus v) = u\cup (v\setminus u)=u\cup v$.

Starting from the definition of $\mathbb{E}_P[f(X_v)\,\vert\,Y_u\in O_u]$ using Bayes' rule, we find
\begin{align*}
\mathbb{E}_P[f(X_v)\,\vert\,Y_u\in O_u] &\coloneqq \sum_{x_v\in\states_v} f(x_v)\frac{P(X_v=x_v, Y_u\in O_u)}{P(Y_u\in O_u)} \\
&= \sum_{x_v\in\states_v} \sum_{x_w\in\states_w} f(x_v)\frac{P(X_v=x_v, X_w=x_w, Y_u\in O_u)}{P(Y_u\in O_u)} \\
&= \sum_{x_u\in\states_u}\sum_{x_{v\setminus u}\in\states_{v\setminus u}} f(x_v)\frac{P(X_u=x_u, X_{v\setminus u}=x_{v\setminus u}, Y_u\in O_u)}{P(Y_u\in O_u)} \\
&= \sum_{x_u\in\states_u}\sum_{x_{v\setminus u}\in\states_{v\setminus u}} f(x_v)\frac{P_{\observs\vert\states}(O_u\vert x_u)P_\states(X_u=x_u, X_{v\setminus u}=x_{v\setminus u})}{P(Y_u\in O_u)}\,,
\end{align*}
where the first equality is by definition, the second is by the basic rules of probability, the third is by changing the indexing using $v\cup w=u\cup(v\setminus u)$, and the fourth is by the definition of augmented stochastic processes in Section~\ref{sec:aug_stochastic_processes}. Joining the indexes using $u\cup(v\setminus u)=u\cup v$, we find that
\begin{align*}
\mathbb{E}_P[f(X_v)\,\vert\,Y_u\in O_u] &= \sum_{x_u\in\states_u}\sum_{x_{v\setminus u}\in\states_{v\setminus u}} f(x_v)\frac{P_{\observs\vert\states}(O_u\vert x_u)P_\states(X_u=x_u, X_{v\setminus u}=x_{v\setminus u})}{P(Y_u\in O_u)} \\
 &= \sum_{x_{u\cup v}\in\states_{u\cup v}} f(x_v)\frac{P_{\observs\vert\states}(O_u\vert x_u)P_\states(X_{u\cup v}=x_{u\cup v})}{P(Y_u\in O_u)} \\
 &= \frac{\mathbb{E}_{P_\states}[f(X_v)P_{\observs\vert\states}(O_u\vert X_u)]}{P(Y_u\in O_u)}\,,
\end{align*}
where the third equality is by the definition of expectation. Applying Lemma~\ref{lemma:output_probability_is_expectation} to the denominator, we finally obtain
\begin{equation*}
\mathbb{E}_P[f(X_v)\,\vert\,Y_u\in O_u] = \frac{\mathbb{E}_{P_\states}[f(X_v)P_{\observs\vert\states}(O_u\vert X_u)]}{\mathbb{E}_{P_\states}[P_{\observs\vert\states}(O_u\vert X_u)]}\,.
\end{equation*}
\end{proof}

\begin{proof}{\bf of Proposition~\ref{prop:GBR_regular}~}
This is a special case of Corollary~\ref{cor:lower_hidden_is_root_chain}. To see this, set $g(X_u)\coloneqq P_{\observs\vert\states}(O_u\vert X_u)$ in that corollary's statement, apply Proposition~\ref{prop:precise_conditioning_for_positive} to the quantities $\mathbb{E}_P[f(X_v)\,\vert\,Y_u\in O_u]$ and apply Lemma~\ref{lemma:output_probability_is_expectation} to the quantities $P(Y_u\in O_u)$.

The hypothesis $\overline{P}_\mathcal{Z}(Y_u\in O_u)>0$ then implies the hypothesis $\uexp[g(X_u)]>0$ of Corollary~\ref{cor:lower_hidden_is_root_chain}, due to Lemma~\ref{lemma:lower_output_probability_is_expectation}.
\end{proof}

\section{Proofs of the Results in Section~\ref{subsec:uncountable}}

\begin{proof}{\bf of Proposition~\ref{prop:precise_bayes_rule_densities}~}
Assume that there is a sequence $\{\lambda_i\}_{i\in\nats}$ in $\realspos$ such that, for all $x_u\in\states_u$,
\begin{equation*}
\phi_u(y_u\vert x_u) \coloneqq \lim_{i\to+\infty}\frac{P_{\observs\vert\states}(O_u\vert x_u)}{\lambda_i}
\end{equation*}
exists, is real-valued, and satisfies $\mathbb{E}_{P_\states}[\phi_u(y_u\vert X_u)]>0$. 

Then, the existence of $\phi_u(y_u\vert X_u)$ clearly implies that
\begin{equation*}
\lim_{i\to+\infty} \frac{f(X_v)P_{\observs\vert\states}(O_u^i\vert X_u)}{\lambda_i} = f(X_v)\phi_u(y_u\vert X_u)\,,
\end{equation*}
and so, using Lemma~\ref{lemma:limit_exp_is_exp_limit}, we find that
\begin{equation*}
\lim_{i\to+\infty}\frac{\mathbb{E}_{P_\states}[f(X_v)P_{\observs\vert\states}(O_u^i\vert X_u)]}{\lambda_i} = \mathbb{E}_{P_\states}[f(X_v)\phi_u(y_u\vert X_u)]
\end{equation*}
exists, and similarly that
\begin{equation*}
\lim_{i\to+\infty}\frac{\mathbb{E}_{P_\states}[P_{\observs\vert\states}(O_u^i\vert X_u)]}{\lambda_i} = \mathbb{E}_{P_\states}[\phi_u(y_u\vert X_u)] > 0\,.
\end{equation*}
Furthermore, this latter inequality implies that there is some $n\in\nats$ such that for all $i>n$,
\begin{equation*}
\frac{\mathbb{E}_{P_\states}[P_{\observs\vert\states}(O_u^i\vert X_u)]}{\lambda_i} > 0\,,
\end{equation*}
and hence in particular $\mathbb{E}_{P_\states}[P_{\observs\vert\states}(O_u^i\vert X_u)]>0$ for all $i>n$. Furthermore, since by construction $O_u^j\supseteq O_u^{j+1}$ for all $j\in\nats$, monotonicity of the measure $P_{\observs\vert\states}$ implies that $P_{\observs\vert\states}(O_u^j\vert x_u) \geq P_{\observs\vert\states}(O_u^{j+1}\vert x_u)$ for all $j\in\nats$ and all $x_u\in\states_u$. So, for all $j\leq n$ we also have $\mathbb{E}_{P_\states}[P_{\observs\vert\states}(O_u^j\vert X_u)]>0$, and so we have found that $\mathbb{E}_{P_\states}[P_{\observs\vert\states}(O_u^i\vert X_u)]>0$ for all $i\in\nats$. Due to Proposition~\ref{prop:precise_conditioning_for_positive}, this implies that each $\mathbb{E}_P[f(X_v)\,\vert\,Y_u\in O_u^i]$ is well-defined.

By the limit definition of $\mathbb{E}_P[f(X_v)\,\vert\,Y_u=y_u]$, and applying Proposition~\ref{prop:precise_conditioning_for_positive} to each step,
\begin{align*}
\mathbb{E}_P[f(X_v)\,\vert\,Y_u=y_u] &= \lim_{i\to+\infty} \mathbb{E}_P[f(X_v)\,\vert\,Y_u\in O_u^i] \\
 &= \lim_{i\to+\infty} \frac{\mathbb{E}_{P_\states}[f(X_v)P_{\observs\vert\states}(O_u^i\vert X_u)]}{\mathbb{E}_{P_\states}[P_{\observs\vert\states}(O_u^i\vert X_u)]} \\
 &= \lim_{i\to+\infty} \frac{\lambda_i}{\lambda_i}\frac{\mathbb{E}_{P_\states}[f(X_v)P_{\observs\vert\states}(O_u^i\vert X_u)]}{\mathbb{E}_{P_\states}[P_{\observs\vert\states}(O_u^i\vert X_u)]} \\
 &= \lim_{i\to+\infty} \frac{\nicefrac{\mathbb{E}_{P_\states}[f(X_v)P_{\observs\vert\states}(O_u^i\vert X_u)]}{\lambda_i}}{\nicefrac{\mathbb{E}_{P_\states}[P_{\observs\vert\states}(O_u^i\vert X_u)]}{\lambda_i}} \\
 &= \frac{\lim_{i\to+\infty} \nicefrac{\mathbb{E}_{P_\states}[f(X_v)P_{\observs\vert\states}(O_u^i\vert X_u)]}{\lambda_i}}{\lim_{i\to+\infty} \nicefrac{\mathbb{E}_{P_\states}[P_{\observs\vert\states}(O_u^i\vert X_u)]}{\lambda_i}}= \frac{\mathbb{E}_{P_\states}[f(X_v)\phi_u(y_u\vert X_u)]}{\mathbb{E}_{P_\states}[\phi_u(y_u\vert X_u)]},
\end{align*}
using the above established existence and properties of the individual limits for the penultimate step.

\end{proof}

The below proves some of the properties that are claimed in the main text of Section~\ref{subsec:uncountable}. We first need the following result, which is essentially well-known, but which we repeat here for the sake of completeness.

\begin{lemma}\label{lemma:lebesgue_differentiation_theorem}
Fix $d\in\nats$ and consider any absolutely integrable function $\psi:\reals^d\to \reals$. Then, for any $y\in\reals^d$ and any sequence $\{B_i\}_{i\in\nats}$ of open balls that are centred on, and shrink to, $y$, if $\psi$ is continuous at $y$ it holds that
\begin{equation*}
\psi(y) = \lim_{i\to+\infty} \frac{1}{\lambda(B_i)}\int_{B_i}\psi(\gamma)\,\mathrm{d}\gamma\,,
\end{equation*}
where the integral is understood in the Lebesgue sense, and where $\lambda(B_i)$ denotes the Lebesgue measure of $B_i$.
\end{lemma}
\begin{proof}
Fix any $\epsilon>0$. We need to show that there is some $n\in\nats$ such that, for all $i>n$, it holds that
\begin{equation*}
\abs{\psi(y) - \frac{1}{\lambda(B_i)}\int_{B_i}\psi(\gamma)\,\mathrm{d}\gamma} < \epsilon\,.
\end{equation*}
Now, because $\psi$ is continuous at $y$, there is some open ball $B$ that is centred on $y$, such that for all $\gamma\in B$, it holds that
\begin{equation*}
\abs{\psi(y)-\psi(\gamma)} \leq \epsilon\,.
\end{equation*}
Furthermore, because the sequence $\{B_i\}_{i\in\nats}$ is centred on, and shrinks to, $y$, there must be some $n\in\nats$ such that for all $i>n$, it holds that $B_i\subset B$. Fix any such $i>n$. Then,
\begin{align*}
\abs{\frac{1}{\lambda(B_i)}\int_{B_i}\psi(\gamma)\,\mathrm{d}\gamma - \psi(y)} &= \abs{\frac{1}{\lambda(B_i)}\int_{B_i}\psi(\gamma) - \psi(y)\,\mathrm{d}\gamma} \\
 &\leq \frac{1}{\lambda(B_i)}\int_{B_i}\abs{\psi(\gamma) - \psi(y)}\,\mathrm{d}\gamma \\
 &\leq \frac{1}{\lambda(B_i)}\int_{B_i}\epsilon\,\mathrm{d}\gamma
 = \frac{\epsilon}{\lambda(B_i)}\int_{B_i}1\,\mathrm{d}\gamma
 = \frac{\epsilon}{\lambda(B_i)}\lambda(B_i)=\epsilon.
\end{align*}
\end{proof}

\begin{proof}{\bf of claims in Section~\ref{subsec:uncountable}~}
We start by proving the claim that $\phi_u(y_u\vert x_u)$ exists and is real-valued if it can be constructed ``piecewise'', as explained in the main text. So, fix any $y_u\in\observs_u$, choose any sequence $\{O_u^i\}_{i\in\nats}$ in $\Sigma_u$  that shrinks to $y_u$, and suppose that for every $t\in u$, there is a sequence $\{\lambda_{\,t,i}\}_{i\in\nats}$ in $\realspos$ such that, for all $x_t\in\states_t$,
\begin{equation*}
\phi_t(y_t\vert x_t) \coloneqq \lim_{i\to+\infty} \frac{P_{\observs\vert\states}(O_t^i\vert x_t)}{\lambda_{\,t,i}}
\end{equation*}
exists and is real-valued. Recall that, for every $x_u\in\states_u$ and every $i\in\nats$, we have $P_{\observs\vert\states}(O_u^i\vert x_u)=\prod_{t\in u}P_{\observs\vert\states}(O_t^i\vert x_t)$. So, by choosing $\{\lambda_i\}_{i\in\nats}$ as $\lambda_i\coloneqq \prod_{t\in u}\lambda_{\,t,i}$, it follows that for every $x_u\in\states_u$,
\begin{equation*}
\phi_u(y_u\vert x_u) = \lim_{i\to+\infty} \frac{P_{\observs\vert\states}(O_u^i\vert x_u)}{\lambda_i} = \lim_{i\to+\infty} \prod_{t\in u}\frac{P_{\observs\vert\states}(O_t^i\vert x_t)}{\lambda_{\,t,i}} = \prod_{t\in u}\lim_{i\to+\infty} \frac{P_{\observs\vert\states}(O_t^i\vert x_t)}{\lambda_{\,t,i}} = \prod_{t\in u} \phi_t(y_t\vert x_t)\,,
\end{equation*}
using the existence of the $\phi_t(y_t\vert x_t)$ for the third equality. Hence, $\phi_u(y_u\vert x_u)$ exists and, since each $\phi_t(y_t\vert x_t)$, $t\in u$, is real-valued, so is $\phi_u(y_u\vert x_u)$. This concludes the proof of this statement.

Next, we prove that the limit expression, and in particular the second equality, in Equation~\eqref{eq:density_is_limit} are true when $\psi(\cdot\vert x_t)$ is continuous (at $y_t$). To this end, note that $P_{\observs\vert\states}(\cdot\,\vert x_t)$ was defined by
\begin{equation*}
P_{\observs\vert\states}(O\,\vert x_t) \coloneqq \int_O \psi(y\vert x_t)\,\mathrm{d}y\,,
\end{equation*}
for all $O\in\Sigma$. For a sequence of open balls $\{O_t^i\}_{i\in\nats}$ that are centred on, and shrink to, $y_t$, we therefore need to prove that
\begin{equation*}
\psi(y_t\vert x_t) = \lim_{i\to+\infty}\frac{P_{\observs\vert\states}(O_t^i\vert x_t)}{\lambda(O_t^i)} = \lim_{i\to+\infty}\frac{1}{\lambda(O_t^i)}\int_{O_t^i}\psi(y\vert x_t)\,\mathrm{d} y\,.
\end{equation*}
Because $\psi(y_t\vert x_t)$ is by assumption continuous at $y_t$, this result follows immediately from Lemma~\ref{lemma:lebesgue_differentiation_theorem}.

We finally prove the claim that $\mathbb{E}_P[f(X_v)\,\vert\,Y_u=y_u]$ is the same for almost every sequence $\{O_u^i\}_{i\in\nats}$ that shrinks to $y_u$, provided that for all $x_t\in\states$, $P_{\observs\vert\states}(\cdot\vert x_t)$ is constructed from a density $\psi(\cdot\,\vert x_t)$ that is continuous and strictly positive at $y_t$. To this end, assume that $\{O_u^i\}_{i\in\nats}$ satisfies the assumptions in Footnote~\ref{fnote:limit_required_properties}; i.e. that for all $t\in u$, there is a sequence of open balls $\{B_t^i\}_{i\in\nats}$ in $\observs$ that shrinks to $y_t$ such that, for all $i\in\nats$, $\lambda(O_t^i)>0$ and $O_t^i\subseteq B_t^i$.

We start by showing that
\begin{equation*}
\phi_u(y_u\vert x_u) = \lim_{i\to+\infty} \frac{P_{\observs\vert\states}(O_u^i\vert x_u)}{\lambda(O_u^i)}
\end{equation*}
is independent of the sequence $\{O_u^i\}_{i\in\nats}$. We will prove this ``piecewise''. In particular, we will show that for all $t\in u$,
\begin{equation*}
\psi(y_t\vert x_t) = \lim_{i\to+\infty} \frac{P_{\observs\vert\states}(O_t^i\vert x_t)}{\lambda(O_t^i)}\,,
\end{equation*}
where we use the assumption $\lambda(O_t^i)>0$ to ensure that each element of the sequence is well-defined.

So, consider any $t\in u$. Note that we have $P_{\observs\vert\states}(O_t^i\vert x_t)\coloneqq \int_{O_t^i}\psi(y\vert x_t)\,\mathrm{d}y$ by definition, and choose any $\epsilon\in\realspos$. Because $\psi(\cdot\vert x_t)$ is continuous at $y_t$, there is some open ball $B_*\in\observs$ that is centred on $y_t$, and such that for all $y\in B_*$,
\begin{equation*}
\abs{\psi(y_t\vert x_t) - \psi(y\vert x_t)} \leq \epsilon\,.
\end{equation*}
Furthermore, because the sequence $\{B_t^i\}_{i\in\nats}$ shrinks to $y_t$, there is some $n\in\nats$ such that, for all $i>n$, it holds that $B_t^i\subseteq B_*$. Furthermore, because each $O_t^i\subseteq B_t^i$, also clearly $O_t^i\subseteq B_*$ for all $i>n$. Now consider any $i>n$. Then, 
\begin{align*}
\abs{\psi(y_t\vert x_t) - \frac{P_{\observs\vert\states}(O_t^i\vert x_t)}{\lambda(O_t^i)}} &= \abs{\psi(y_t\vert x_t) - \frac{1}{\lambda(O_t^i)}\int_{O_t^i}\psi(y\vert x_t)\,\mathrm{d}y} \\
 &\leq \frac{1}{\lambda(O_t^i)}\int_{O_t^i}\abs{\psi(y\vert x_t) - \psi(y_t\vert x_t)}\,\mathrm{d}y \\
 &\leq \frac{1}{\lambda(O_t^i)}\int_{O_t^i}\epsilon\,\mathrm{d}y
 = \frac{\epsilon}{\lambda(O_t^i)}\int_{O_t^i}1\,\mathrm{d}y 
 = \frac{\epsilon}{\lambda(O_t^i)}\lambda(O_t^i)
 = \epsilon\,.
\end{align*}
So, we conclude that indeed
\begin{equation*}
\psi(y_t\vert x_t) = \lim_{i\to+\infty}\frac{P_{\observs\vert\states}(O_t^i\vert x_t)}{\lambda(O_t^i)}\,,
\end{equation*}
as claimed. Because this holds for all $t\in u$, it follows from what we discussed above that $\phi_u(y_u\vert x_u)$ can be constructed ``piecewise'', that is,
\begin{equation*}
\phi_u(y_u\vert x_u) = \lim_{i\to+\infty}\frac{P_{\observs\vert\states}(O_u^i\vert x_u)}{\lambda(O_u^i)} = \lim_{i\to+\infty}\prod_{t\in u}\frac{P_{\observs\vert\states}(O_t^i\vert x_t)}{\lambda(O_t^i)} = \prod_{t\in u}\psi(y_t\vert x_t)\,,
\end{equation*}
which implies that $\phi_u(y_u\vert x_u)$ exists and is the same for every sequence $\{O_u^i\}_{i\in\nats}$ for which the assumed properties hold. Furthermore, since by assumption each $\psi(y_t\vert x_t)>0$, we clearly also have $\phi_u(y_u\vert x_u)>0$ for all $x_u\in\states_u$, and therefore in particular that $\mathbb{E}_{P_\states}[\phi_u(y_u\vert X_u)]>0$. 

It now follows from Proposition~\ref{prop:precise_bayes_rule_densities} that $\mathbb{E}_P[f(X_v)\,\vert\,Y_u=y_u]$ exists and is the same for every sequence for which the assumed properties hold.
\end{proof}

\begin{proof}{\bf of Proposition~\ref{prop:GBR_for_densities_lower_zero}~}
This is a special case of Corollary~\ref{cor:lower_hidden_is_root_chain}, obtained by setting $g(X_u)\coloneqq \phi_u(y_u\vert X_u)$ in that corollary's statement, and applying Proposition~\ref{prop:precise_bayes_rule_densities} to the quantities $\mathbb{E}_P[f(X_v)\,\vert\,Y_u=y_u]$.
\end{proof}

\begin{proof}{\bf of Proposition~\ref{prop:GBR_for_densities_is_limit_if_continuous}~}
Assume that there is a sequence $\{\lambda_i\}_{i\in\nats}$ such that $\phi_u(y_u\vert X_u)$ exists, is real-valued, and satisfies $\lexp[\phi_u(y_u\vert X_u)] >0$. 

Note that $\lexp[\phi_u(y_u\vert X_u)] >0$ implies that $\mathbb{E}_{P_\states}[\phi_u(y_u\vert X_u)]>0$ for all $P_\states\in\mathbb{P}_{\rateset,\mathcal{M}}$. Furthermore, for any $P_\states\in\mathbb{P}_{\rateset,\mathcal{M}}$, because by assumption $\lim_{i\to+\infty}\nicefrac{P_{\observs\vert\states}(O_u^i\vert X_u)}{\lambda_i}=\phi_u(y_u\vert X_u)$, it follows from Lemma~\ref{lemma:limit_exp_is_exp_limit} that
\begin{equation*}
\lim_{i\to+\infty}\frac{\mathbb{E}_{P_\states}[P_{\observs\vert\states}(O_u^i\vert X_u)]}{\lambda_i} = \mathbb{E}_{P_\states}[\phi_u(y_u\vert X_u)] > 0\,.
\end{equation*}
This implies that there is some $n\in\nats$ such that for all $j>n$,
\begin{equation*}
\frac{\mathbb{E}_{P_\states}[P_{\observs\vert\states}(O_u^j\vert X_u)]}{\lambda_j} > 0\,,
\end{equation*}
which implies that also $\mathbb{E}_{P_\states}[P_{\observs\vert\states}(O_u^j\vert X_u)] >0$. Furthermore, since $O_u^i\supseteq O_u^{i+1}$, we have that $P_{\observs\vert\states}(O_u^i\vert x_u)\geq P_{\observs\vert\states}(O_u^{i+1}\vert x_u)$ for all $x_u\in\states_u$, by monotonicity of measure. It follows that also for all $k\leq n$
\begin{equation*}
\mathbb{E}_{P_\states}[P_{\observs\vert\states}(O_u^k\vert X_u)] \geq \mathbb{E}_{P_\states}[P_{\observs\vert\states}(O_u^j\vert X_u)] > 0\,.
\end{equation*}
Hence, we have found that $\mathbb{E}_{P_\states}[P_{\observs\vert\states}(O_u^i\vert X_u)]>0$ for all $i\in\nats$. Since $P\in\mathbb{P}_{\rateset,\mathcal{M}}$, it now follows that $\uexp[P_{\observs\vert\states}(O_u^i\vert X_u)]=\overline{P}_{\mathcal{Z}}(Y_u\in O_u^i)>0$ for all $i\in\nats$ by Lemma~\ref{lemma:lower_output_probability_is_expectation}.

Now define the sequence $\{\mu_i\}_{i\in\nats}$ as $\mu_i\coloneqq \underline{\mathbb{E}}_\mathcal{Z}[f(X_v)\,\vert\,Y_u\in O_u^i]$, for all $i\in\nats$.
Fix any $i\in\nats$. Then, because $\overline{P}_{\mathcal{Z}}(Y_u\in O_u^i)>0$, it follows from Definition~\ref{def:reg_ext_pos} that $\mu_i\in[\min f, \max f]$ and furthermore, by Proposition~\ref{prop:GBR_regular}, that
\begin{equation*}
\lexp[P_{\observs\vert\states}(O_u^i\vert X_u)\bigl(f(X_v) - \mu_i\bigr)] = 0\,.
\end{equation*}
Therefore, and by the non-negative homogeneity of lower expectations, it also holds that
\begin{equation}\label{eq:nat_ext_limit:steps_are_roots}
\lexp\left[\frac{P_{\observs\vert\states}(O_u^i\vert X_u)}{\lambda_i}\bigl(f(X_v) - \mu_i\bigr)\right] = 0\,.
\end{equation}

Let now $\{\mu_{i_k}\}_{k\in\nats}$ be any convergent subsequence; since the sequence $\{\mu_i\}_{i\in\nats}$ is in the compact interval $[\min f,\max f]$, the Bolzano-Weierstrass theorem implies that at least one such subsequence exists. Let $\mu_{*}\coloneqq \lim_{k\to+\infty}\mu_{i_k}$.

We now clearly have that
\begin{equation*}
\lim_{k\to+\infty} \frac{P_{\observs\vert\states}(O_u^{i_k}\vert X_u)}{\lambda_{i_k}}\bigl(f(X_v) - \mu_{i_k}\bigr) = \phi_u(y_u\vert X_u)\bigl(f(X_v) - \mu_{*}\bigr)\,,
\end{equation*}
and therefore, by Lemma~\ref{lemma:limit_lexp_is_lexp_limit}, that
\begin{equation*}
\lim_{k\to+\infty} \lexp\left[\frac{P_{\observs\vert\states}(O_u^{i_k}\vert X_u)}{\lambda_{i_k}}\bigl(f(X_v) - \mu_{i_k}\bigr)\right] = \lexp\left[\phi_u(y_u\vert X_u)\bigl(f(X_v) - \mu_{*}\bigr)\right]\,.
\end{equation*}
Furthermore, using Equation~\eqref{eq:nat_ext_limit:steps_are_roots}, we find that $\lexp\left[\phi_u(y_u\vert X_u)\bigl(f(X_v) - \mu_{*}\bigr)\right] = 0$.

Due to Proposition~\ref{prop:GBR_properties}, and because $\lexp[\phi_u(y_u\vert X_u)]>0$, we conclude that $\mu_{*}$ corresponds to the unique root of the function $G(\mu)\coloneqq \lexp\left[\phi_u(y_u\vert X_u)\bigl(f(X_v) - \mu\bigr)\right]$. Furthermore, since the convergent subsequence $\{\mu_{i_k}\}_{k\in\nats}$ was arbitrary, we find that $\mu_*$ is the limit of \emph{every} convergent subsequence of $\{\mu_i\}_{i\in\nats}$.

We next show that $\{\mu_i\}_{i\in\nats}$ itself also converges to $\mu_*$. Assume \emph{ex absurdo} that this is false. Then, there is some $\epsilon>0$ such that, for all $n\in\nats$, there is some $k>n$ such that $\abs{\mu_k - \mu_*} \geq \epsilon$. This implies that we can construct a subsequence $\{\mu_{i_k}\}_{k\in\nats}$ such that $\abs{\mu_{i_k}-\mu_*}\geq \epsilon$ for all $k\in\nats$. This subsequence is again in the compact interval $[\min f, \max f]$, which implies that it has a convergent subsequence $\{\mu_{i_{k_\ell}}\}_{\ell\in\nats}$, and clearly $\lim_{\ell\to+\infty}\mu_{i_{k_\ell}} \neq \mu_*$ because $\abs{\mu_{i_{k_\ell}} - \mu_*}\geq\epsilon$ for all $\ell\in\nats$. However, since $\{\mu_{i_{k_\ell}}\}_{\ell\in\nats}$ is a convergent subsequence of the original sequence $\{\mu_i\}_{i\in\nats}$, this contradicts our above conclusions. Hence, we must have that indeed $\lim_{i\to\infty}\mu_i=\mu_*$.

Since we already know that $\mu_*$ is the unique root of the function $G(\mu)$ defined above, and since this function is strictly decreasing due to Proposition~\ref{prop:GBR_properties}, we must have that
\begin{equation*}
\mu_* = \max\{\mu\in\reals\,:\, G(\mu)\geq 0\}\,.
\end{equation*}
Therefore, due to Proposition~\ref{prop:GBR_for_densities_lower_zero}, we conclude that
\begin{align*}
\underline{\mathbb{E}}_\mathcal{Z}[f(X_v)\,\vert\,Y_u=y_u] &= \max\left\{\mu\in\reals\,:\, \lexp\left[\phi_u(y_u\vert X_u)\bigl(f(X_v) - \mu\bigr)\right] \geq 0\right\} \\
 &= \mu_* = \lim_{i\to+\infty} \mu_i 
 = \lim_{i\to+\infty} \underline{\mathbb{E}}_\mathcal{Z}[f(X_v)\,\vert\,Y_u\in O_u^i]\,.
\end{align*}
\end{proof}

The following two propositions prove the properties of the first of the two alternative imprecise updating methods that were suggested in Section~\ref{subsec:uncountable}.
\begin{proposition}\label{prop:regular_is_computable}
Let $\mathcal{Z}$ be an ICTHMC and consider any $u,v\in\mathcal{U}$, $y_u\in\observs_u$ and $f\in\gambles(\states_v)$. For any $\{O_u^i\}_{i\in\nats}$ in $\Sigma_u$ that shrinks to $y_u$, if for some $\{\lambda_i\}_{i\in\nats}$ in $\realspos$ the quantity $\phi_u(y_u\vert X_u)$ exists, is real-valued and satisfies $\uexp[\phi_u(y_u\vert X_u)]>0$, then the model defined by
\begin{equation*}
\underline{\mathbb{E}}_\mathcal{Z}^\mathrm{R}[f(X_v)\,\vert\,Y_u=y_u] \coloneqq \inf\left\{\mathbb{E}_P[f(X_v)\,\vert\,Y_u=y_u]\,:\,P\in\mathcal{Z},\,\mathbb{E}_{P_\states}[\phi_u(y_u\vert X_u)]>0\right\}\,,
\end{equation*}
satisfies
\begin{equation*}
\underline{\mathbb{E}}_\mathcal{Z}^\mathrm{R}[f(X_v)\,\vert\,Y_u=y_u] = \max\{\mu\in\reals\,:\,\lexp\bigl[\phi_u(y_u\vert X_u)\bigl(f(X_v)-\mu\bigr)\bigr] \geq 0\}\,.
\end{equation*}
\end{proposition}
\begin{proof}
This is a special case of Corollary~\ref{cor:lower_hidden_is_root_chain}, obtained by setting $g(X_u)\coloneqq \phi_u(y_u\vert X_u)$ in that corollary's statement, and applying Proposition~\ref{prop:precise_bayes_rule_densities} to the quantities $\mathbb{E}_P[f(X_v)\,\vert\,Y_u=y_u]$.
\end{proof}

\begin{proposition}\label{prop:counter_example_regular}
There exists some ICTHMC $\mathcal{Z}$, some sequences of time-points $u,v\in\mathcal{U}$, some function $f\in\gambles(\states_v)$, some $y_u\in\observs_u$ and some sequence $\{O_u^i\}_{i\in\nats}$ such that there is a sequence $\{\lambda_i\}_{i\in\nats}$ for which $\phi_u(y_u\vert X_u)$ exists, is real-valued and satisfies $\uexp[\phi_u(y_u\vert X_u)]>0$, such that for
\begin{equation*}
\underline{\mathbb{E}}_\mathcal{Z}^\mathrm{R}[f(X_v)\,\vert\, Y_u=y_u] \coloneqq \inf\left\{ \mathbb{E}_P[f(X_v)\,\vert\,Y_u=y_u]\,:\,P\in\mathcal{Z}, \mathbb{E}_P[\phi_u(y_u\vert X_u)]>0  \right\}\,,
\end{equation*}
it holds that
\begin{equation*}
\underline{\mathbb{E}}_\mathcal{Z}^\mathrm{R}[f(X_v)\,\vert\, Y_u=y_u] \neq \lim_{i\to+\infty} \underline{\mathbb{E}}_\mathcal{Z}[f(X_v)\,\vert\, Y_u\in O_u^i]\,.
\end{equation*}
\end{proposition}
\begin{proof}
Because the claim is existential, a proof by example suffices. To this end, let $\states\coloneqq\{x,\overline{x}\}$ be a binary state-space, and let $\observs\coloneqq [-1,1]$, with $\Sigma$ the Borel $\sigma$-algebra on $\observs$ under the usual topology. For the sequences of time-points, we choose $u=v=\{0\}$. Set $y_0\coloneqq 0$.

The trick will be to choose the measures $P_{\observs\vert\states}(\cdot\vert x)$ and $P_{\observs\vert\states}(\cdot\vert \overline{x})$ so that the first gives $y_0$ strictly positive support (i.e. density), while the second assigns zero support (i.e. density) to $y_0$ but positive support (i.e. density) to the region around $y_0$. To this end, let $P_{\observs\vert\states}(\cdot\vert x)$ be the uniform distribution on $[-1,1]$.

We define the measure $P_{\observs\vert\states}(\cdot\vert \overline{x})$ by constructing $\phi(\cdot\vert\overline{x})$ explicitly. So, for every $y\in\observs$, let
\begin{equation*}
\phi(y\vert\overline{x}) \coloneqq \abs{y}\,.
\end{equation*}
Clearly, $\int_\observs \phi(y\vert\overline{x})\,\mathrm{d}y = 1$, and furthermore $\phi(y_0\vert\overline{x})=0$. For every $O\in\Sigma$, let $P_{\observs\vert\states}(O\vert\overline{x})$ be defined by
\begin{equation*}
P_{\observs\vert\states}(O\vert\overline{x}) \coloneqq \int_O\phi(y\vert\overline{x})\,\mathrm{d}y\,.
\end{equation*}
For any sequence $\{O_0^i\}_{i\in\nats}$ of open intervals $O_0^i\coloneqq (y_0-\delta_i,y_0+\delta_i)$, with $\delta_i>0$ such that $\{\delta_i\}_{i\in\nats}\to0^+$, we then clearly have
\begin{equation*}
\phi(y_0\vert x) = \lim_{i\to+\infty} \frac{P_{\observs\vert\states}(O_0^i\vert x)}{\lambda(O_0^i)} = \lim_{i\to+\infty} \frac{\delta_i}{2\delta_i} = \frac{1}{2}\,,
\end{equation*}
and
\begin{equation*}
\phi(y_0\vert \overline{x}) = \lim_{i\to+\infty} \frac{P_{\observs\vert\states}(O_0^i\vert \overline{x})}{\lambda(O_0^i)} = \lim_{i\to+\infty} \frac{\delta_i^2}{2\delta_i} = \lim_{i\to+\infty} \frac{\delta_i}{2} = 0\,,
\end{equation*}
where $\lambda(O_0^i)$ is the Lebesgue measure of $O_0^i$. Fix any such sequence $\{O_0^i\}_{i\in\nats}$ and let $\{\lambda_i\}_{i\in\nats}\coloneqq\{\lambda(O_0^i)\}_{i\in\nats}$.

Let the set of initial distributions $\mathcal{M}$ be the entire set of probability mass functions on $\states_0$. Consider the two probability mass functions $P_1,P_2$ on $\states_0$ such that $P_1(x)=1$, $P_1(\overline{x})=0$, and $P_2(x)=0$, $P_2(\overline{x})=1$; clearly, $P_1,P_2\in\mathcal{M}$. Let $\rateset$ be any non-empty, bounded, and convex set of rate matrices with separately specified rows, and let $\mathbb{P}_{\rateset,\mathcal{M}}$ be the corresponding ICTMC. Construct $\mathcal{Z}$ from $\mathbb{P}_{\rateset,\mathcal{M}}$ and $(\observs,\Sigma,P_{\observs\vert\states})$ as in Definition~\ref{def:hidden_ictmc}. Let $f\in\gambles(\states_0)$ be defined by $f(x)\coloneqq 1$ and $f(\overline{x})\coloneqq -1$. 

Clearly, $\uexp[\phi(y_0\vert X_0)]=\overline{\mathbb{E}}_\mathcal{M}[\phi(y_0\vert X_0)] = P_1(x)\phi(y_0\vert x) + P_1(\overline{x})\phi(y_0\vert \overline{x})=P_1(x)\phi(y_0\vert x) = \frac{1}{2}>0$, and so the lower expectation $\underline{\mathbb{E}}_\mathcal{Z}^\mathrm{R}[f(X_v)\,\vert\, Y_u=y_u]$ is well-defined.

Now, for any $P\in\mathcal{Z}$ such that $\mathbb{E}_{P}[\phi(y_0\vert X_0)]>0$, we have by Proposition~\ref{prop:precise_bayes_rule_densities} that
\begin{equation*}
\mathbb{E}_{P}[f(X_0)\,\vert\,Y_0=y_0] = \frac{\mathbb{E}_{P}[f(X_0)\phi(y_0\vert X_0)]}{\mathbb{E}_{P}[\phi(y_0\vert X_0)]} = \frac{P(x)f(x)\phi(y_0\vert x)}{P(x)\phi(y_0\vert x)} = f(x) = 1\,,
\end{equation*}
because $\phi(y_0\vert \overline{x})=0$. Hence,
\begin{align*}
\underline{\mathbb{E}}_\mathcal{Z}^\mathrm{R}[f(X_0)\,\vert\, Y_0=y_0] &= \inf\left\{ \mathbb{E}_{P}[f(X_0)\,\vert\,Y_0=y_0]\,:\,P\in\mathcal{Z}, \mathbb{E}_{P}[\phi(y_0\vert X_0)]>0  \right\} = 1\,.
\end{align*}
However, for any $i\in\nats$, since $P_2\in\mathcal{M}$, there is some $P\in\mathcal{Z}$ such that 
\begin{equation*}
P(Y_0\in O_0^i)=\mathbb{E}_{P_\states}[P_{\observs\vert\states}(O_0^i\vert X_0)]=\mathbb{E}_{P_2}[P_{\observs\vert\states}(O_0^i\vert X_0)]=P_{\observs\vert\states}(O_0^i\vert \overline{x})>0,
\end{equation*}
and so
\begin{align*}
\mathbb{E}_{P}[f(X_0)\,\vert\,Y_0\in O_0^i] = \frac{\mathbb{E}_{P_\states}[f(X_0)P_{\observs\vert\states}(O_0^i\vert X_0)]}{\mathbb{E}_{P_\states}[P_{\observs\vert\states}(O_0^i\vert X_0)]} = \frac{f(\overline{x})P_{\observs\vert\states}(O_0^i\vert \overline{x})}{P_{\observs\vert\states}(O_0^i\vert \overline{x})} = f(\overline{x}) = -1.
\end{align*}
Since we also know that $-1=\min f \leq \underline{\mathbb{E}}_\mathcal{Z}[f(X_0)\vert Y_0\in O_0^i] \leq \mathbb{E}_{P}[f(X_0)\,\vert\,Y_0\in O_0^i]$, this allows us to infer that $\underline{\mathbb{E}}_\mathcal{Z}[f(X_0)\vert Y_0\in O_0^i]=-1$.

We conclude from all of the above that
\begin{equation*}
-1 =  \lim_{i\to+\infty} \underline{\mathbb{E}}_\mathcal{Z}[f(X_0)\vert Y_0\in O_0^i] \neq \underline{\mathbb{E}}_\mathcal{Z}^\mathrm{R}[f(X_0)\,\vert\, Y_0=y_0] = 1\,.
\end{equation*}
\end{proof}

The following result proves the claim that the \emph{second} alternative imprecise updating method from Section~\ref{subsec:uncountable} does not in general satisfy the generalised Bayes' rule for mixtures of densities.
\begin{corollary}
There exists some ICTHMC $\mathcal{Z}$, some sequences of time-points $u,v\in\mathcal{U}$, some function $f\in\gambles(\states_v)$, some $y_u\in\observs_u$ and some sequence $\{O_u^i\}_{i\in\nats}$ such that
\begin{equation*}
\{P\in\mathcal{Z}\,:\,\text{$\mathbb{E}_P[f(X_v)\,\vert\,Y_u=y_u]$ exists}\}\neq\emptyset
\end{equation*}
and
\begin{equation*}
\underline{\mathbb{E}}_\mathcal{Z}^\mathrm{L}[f(X_v)\,\vert\, Y_u=y_u] \neq \max\{\mu\in\reals\,:\,\lexp\bigl[\phi_u(X_u)\bigl(f(X_v)-\mu\bigr)\bigr]\geq 0\},
\end{equation*}
with
\begin{equation*}
\underline{\mathbb{E}}_\mathcal{Z}^\mathrm{L}[f(X_v)\,\vert\, Y_u=y_u]\coloneqq \inf\bigl\{\mathbb{E}_P[f(X_v)\,\vert\,Y_u=y_u]\,:\,P\in\mathcal{Z},\,\text{$\mathbb{E}_P[f(X_v)\,\vert\,Y_u=y_u]$ exists}\bigr\}
\end{equation*}
\end{corollary}
\begin{proof}
The claim is existential, so a proof by example suffices. Let $\mathcal{Z}$, $u,v\in\mathcal{U}$, $y_u\in\observs_u$, $f\in\gambles(\states_v)$ and $\{O_u^i\}_{i\in\nats}$ in $\Sigma_u$ be the same as in the example/proof of Proposition~\ref{prop:counter_example_regular}.

Now consider any $P\in\mathcal{Z}$, and note that clearly either $\mathbb{E}_{P_\states}[\phi(y_0\vert X_0)]>0$, or $\mathbb{E}_{P_\states}[\phi(y_0\vert X_0)]=0$. For the first case, as we established in the example/proof of Proposition~\ref{prop:counter_example_regular}, we know that $\mathbb{E}_P[f(X_0)\vert Y_0=y_0]$ exists and is equal to $1$.

Suppose for the other case that $\mathbb{E}_{P_\states}[\phi(y_0\vert X_0)]=0$. Because $\phi(y_0\vert x)>0$ and $\phi(y_0\vert \overline{x})=0$, this clearly implies that $P_\states(X_0=x)=0$, or in other words, that $P_\states(X_0=\overline{x})=1$. As we already established in the previous example, we then have for every $i\in\nats$ that
\begin{equation*}
\mathbb{E}_P[f(X_0)\vert Y_0\in O_0^i] = \frac{f(\overline{x})P_{\observs\vert\states}(O_0^i\vert \overline{x})}{P_{\observs\vert\states}(O_0^i\vert \overline{x})} = -1\,.
\end{equation*}
Because this holds for all $\in\nats$, we clearly have that $\mathbb{E}_P[f(X_0)\vert Y_0 = y_0]$ exists, and
\begin{equation*}
\mathbb{E}_P[f(X_0)\vert Y_0=y_0] = \lim_{i\to+\infty} \mathbb{E}_P[f(X_0)\vert Y_0\in O_0^i] = -1\,.
\end{equation*}
Since this exhaustively covers all cases, we conclude that $\mathbb{E}_P[f(X_0)\vert Y_0=y_0]$ exists for all $P\in\mathcal{Z}$. Therefore,
\begin{equation*}
\{P\in\mathcal{Z}\,:\,\text{$\mathbb{E}_P[f(X_v)\,\vert\,Y_u=y_u]$ exists}\} = \mathcal{Z}\,,
\end{equation*}
which means that $\underline{\mathbb{E}}_\mathcal{Z}^\mathrm{L}[f(X_v)\,\vert\, Y_u=y_u]$ is well-defined.

Furthermore, it follows from the above that, for all $P\in\mathcal{Z}$, we have either $\mathbb{E}_P[f(X_0)\vert Y_0=y_0]=1$, or $\mathbb{E}_P[f(X_0)\vert Y_0=y_0]=-1$. Also, since $P_2\in\mathcal{M}$, there is at least one $P\in\mathcal{Z}$ for which the second case applies. Hence, it follows that
\begin{align*}
\underline{\mathbb{E}}_\mathcal{Z}^\mathrm{L}[f(X_v)\,\vert\, Y_u=y_u] &\coloneqq \inf\bigl\{\mathbb{E}_P[f(X_v)\,\vert\,Y_u=y_u]\,:\,P\in\mathcal{Z},\,\text{$\mathbb{E}_P[f(X_v)\,\vert\,Y_u=y_u]$ exists}\bigr\}= -1.
\end{align*}
However, we also know that, for any $P\in\mathcal{Z}$, if $\mathbb{E}_{P_\states}[\phi(y_0\vert X_0)]>0$, then $\mathbb{E}_{P}[f(X_0)\vert Y_0=y_0]=1$. Since $P_1\in\mathcal{M}$, there is at least one $P\in\mathcal{Z}$ for which this holds. Therefore,
\begin{align*}
1 &= \inf\left\{\mathbb{E}_P[f(X_v)\,\vert\,Y_u=y_u]\,:\,P\in\mathcal{Z},\,\mathbb{E}_{P_\states}[\phi(y_0\vert X_0)]>0\right\} \\
 &= \max\{\mu\in\reals\,:\,\lexp\bigl[\phi_u(X_u)\bigl(f(X_v)-\mu\bigr)\bigr]\geq 0\}\,,
\end{align*}
where the second equality follows from Proposition~\ref{prop:regular_is_computable}. We conclude that indeed
\begin{equation*}
-1 = \underline{\mathbb{E}}_\mathcal{Z}^\mathrm{L}[f(X_v)\,\vert\, Y_u=y_u] \neq 
\max\{\mu\in\reals\,:\,\lexp\bigl[\phi_u(X_u)\bigl(f(X_v)-\mu\bigr)\bigr]\geq 0\} = 1\,.
\end{equation*}
\end{proof}

\section{Proofs of the Results in Section~\ref{sec:inference_algos}}

We need the following lemma for the proof of Proposition~\ref{prop:computing_product_funcs}.
\begin{lemma}\label{lemma:product_func_induction}
For all $i\in\{0,\dots,n\}$, let $g_{t_i}$, $g_{t_i}^+$ and $g_{t_i}^-$ be as defined in Section~\ref{sec:funcs_single_time}. Then for all $i\in\{0,\dots,n\}$:
\begin{equation*}
g_{t_i}^+ = \underline{\mathbb{E}}_{\rateset}\left[\prod_{j=i}^{n}g_{t_j}(X_{t_j})\,\Bigg\vert\,X_{t_i}\right] \quad\quad\text{and} \quad\quad g_{t_i}^- = \overline{\mathbb{E}}_{\rateset}\left[\prod_{j=i}^{n}g_{t_j}(X_{t_j})\,\Bigg\vert\,X_{t_i}\right]\,.
\end{equation*}
\end{lemma}
\begin{proof}
We provide a proof by induction. Clearly, the result is trivial for $i=n$. So, assume that it is true for $i$. We show that it is then also true for $i-1$ (with $i>0$).

We focus on $g_{t_{i-1}}^+$, and consider the two cases in its definition separately. So, consider any $x_{t_{i-1}}\in\states_{t_{i-1}}$. Then, if $g_{t_{i-1}}(x_{t_{i-1}})\geq 0$, we have
\begin{align*}
g_{t_{i-1}}^+(x_{t_{i-1}}) &= g_{t_{i-1}}(x_{t_{i-1}}) \underline{\mathbb{E}}_{\rateset}\left[g_{t_{i}}^+(X_{t_{i}})\,\vert\,X_{t_{i-1}}=x_{t_{i-1}}\right] \\
 &= g_{t_{i-1}}(x_{t_{i-1}})\underline{\mathbb{E}}_{\rateset}\left[\underline{\mathbb{E}}_{\rateset}\left[\prod_{j=i}^{n}g_{t_j}(X_{t_j})\,\Bigg\vert\,X_{t_i}\right]\,\Bigg\vert\,X_{t_{i-1}}=x_{t_{i-1}}\right] \\
 &= g_{t_{i-1}}(x_{t_{i-1}})\underline{\mathbb{E}}_{\rateset}\left[\prod_{j=i}^{n}g_{t_j}(X_{t_j})\,\Bigg\vert\,X_{t_{i-1}}=x_{t_{i-1}}\right] \\ 
 &= \underline{\mathbb{E}}_{\rateset}\left[g_{t_{i-1}}(x_{t_{i-1}})\prod_{j=i}^{n}g_{t_j}(X_{t_j})\,\Bigg\vert\,X_{t_{i-1}}=x_{t_{i-1}}\right]\,,
\end{align*}
where the first equality is by definition, the second is by the induction hypothesis, the third by iterated lower expectation (Lemma~\ref{lemma:iterated_lower}), and the final by the non-negative homogeneity of lower expectations and the assumption that $g_{t_{i-1}}(x_{t_{i-1}})\geq 0$.

For the other case, assume that $g_{t_{i-1}}(x_{t_{i-1}})< 0$. Then,
\begin{align*}
g_{t_{i-1}}^+(x_{t_{i-1}}) &= g_{t_{i-1}}(x_{t_{i-1}})\overline{\mathbb{E}}_{\rateset}\left[g_{t_{i}}^-(X_{t_{i}})\,\vert\,X_{t_{i-1}}=x_{t_{i-1}}\right] \\
&= g_{t_{i-1}}(x_{t_{i-1}})\overline{\mathbb{E}}_{\rateset}\left[ \overline{\mathbb{E}}_{\rateset}\left[\prod_{j=i}^{n}g_{t_j}(X_{t_j})\,\Bigg\vert\,X_{t_i}\right]\,\Bigg\vert\,X_{t_{i-1}}=x_{t_{i-1}}\right] \\
&= g_{t_{i-1}}(x_{t_{i-1}})\overline{\mathbb{E}}_{\rateset}\left[\prod_{j=i}^{n}g_{t_j}(X_{t_j})\,\Bigg\vert\,X_{t_{i-1}}=x_{t_{i-1}}\right] \\
&= -g_{t_{i-1}}(x_{t_{i-1}})\underline{\mathbb{E}}_{\rateset}\left[-\prod_{j=i}^{n}g_{t_j}(X_{t_j})\,\Bigg\vert\,X_{t_{i-1}}=x_{t_{i-1}}\right] \\
&= \underline{\mathbb{E}}_{\rateset}\left[g_{t_{i-1}}(x_{t_{i-1}})\prod_{j=i}^{n}g_{t_j}(X_{t_j})\,\Bigg\vert\,X_{t_{i-1}}=x_{t_{i-1}}\right]\,,
\end{align*}
where the first equality is by definition, the second equality by the induction hypothesis, the third by iterated upper expectation (Lemma~\ref{lemma:iterated_lower} combined with conjugacy), the fourth by conjugacy of upper- and lower expectation, and the final by the non-negative homogeneity of lower expectations and the assumption that $g_{t_{i-1}}(x_{t_{i-1}})<0$.

Since this covers both cases in the definition of $g_{t_{i-1}}^+(x_{t_{i-1}})$, we find that
\begin{equation*}
g_{t_{i-1}}^+(X_{t_{i-1}}) = \underline{\mathbb{E}}_{\rateset}\left[g_{t_{i-1}}(X_{t_{i-1}})\prod_{j=i}^{n}g_{t_j}(X_{t_j})\,\Bigg\vert\,X_{t_{i-1}}\right] = \underline{\mathbb{E}}_{\rateset}\left[\prod_{j={i-1}}^{n}g_{t_j}(X_{t_j})\,\Bigg\vert\,X_{t_{i-1}}\right],
\end{equation*}
which concludes the proof for $g_{t_{i-1}}^+$. The proof for $g_{t_{i-1}}^-$ is completely analogous.
\end{proof}

\begin{proof}{\bf of Proposition~\ref{prop:computing_product_funcs}~}
By combining Lemma~\ref{lemma:product_func_induction} with iterated lower expectation (Lemma~\ref{lemma:iterated_lower}), we find that
\begin{equation*}
\lexp\left[g_{t_0}^+(X_{t_0})\right] = \lexp\left[\underline{\mathbb{E}}_{\rateset}\left[\prod_{j=0}^{n}g_{t_j}(X_{t_j})\,\Bigg\vert\,X_{t_0}\right]\right] = \lexp\left[\prod_{j=0}^{n}g_{t_j}(X_{t_j})\right] = \lexp\left[\prod_{t\in u'}g_{t}(X_{t})\right]\,,
\end{equation*}
and similarly for the upper expectation.
\end{proof}

\end{document}